\documentclass[11pt]{amsart}
\usepackage [latin1]{inputenc}
\usepackage{categorytheory,amssymb,graphicx, comment, stmaryrd}
\expandafter\def\csname opt@stmaryrd.sty\endcsname {only,shortleftarrow,shortrightarrow}
\fullpage
\usepackage{todonotes}
\usepackage{tikz-cd}
\usetikzlibrary{decorations.markings}
\usepackage{amsaddr}
\usepackage[all]{xy}
\newcommand{\dotp}{{{\raisebox{.2ex}{\scalebox{.5}{$\bullet$}}}}}
\usepackage{extpfeil}
\usepackage{fourier}
\usepackage{enumerate}

\DeclareMathOperator{\Ar}{Ar}
\DeclareMathOperator{\coeq}{coeq}

\newcommand{\WaldCat}{\mathbf{WaldCat}}

\newcommand{\sT}{\mathscr{T}}
\newcommand{\redS}{\bar{\mathbb{S}}}
\newcommand{\LP}{L^C}

\newcommand{\defeq}{\stackrel{\mathrm{def}}=}

\newcommand{\EOn}[6]
{\begin{tikzcd}[sep = tiny, ampersand replacement=\&]
	\& {#2} \\
	{#1} \&\& {} \\
	\& {#3}
	\arrow["{#4}", bend left=25, densely dashed, from=2-1, to=1-2]
	\arrow["{#6}", bend left=25, densely dashed,  from=3-2, to=2-1]
	\arrow["{#5}", bend left=25, densely dashed, from=1-2, to=2-3]
	\arrow[bend left=25, dotted, no head, from=2-3, to=3-2]
\end{tikzcd}}

\newcommand{\sma}{\wedge}

\usetikzlibrary{decorations.markings}
\tikzset{
  bottomstyle/.style={rectangle,below,text width=0pt,inner xsep=0pt, inner ysep=2.5pt},
  functormark/.style={postaction={
      decorate,
      decoration={
        markings,
        mark=at position .0 with \node[bottomstyle] {$\scriptstyle #1$};}}}
}

\makeatletter
\renewcommand{\boxed}[1]{\text{\fboxsep=.2em\fbox{\m@th$\displaystyle#1$}}}
\makeatother

\newtheorem{maintheorem}{Theorem}

\theoremstyle{remark}
\newtheorem{notation}[equation]{Notation}

\title{A cofibrant model of Waldhausen $K$-theory}
\author[Ponto]{Kate Ponto}
\author[Zakharevich]{Inna Zakharevich}

\begin{document}
\maketitle

\begin{abstract}
  This paper gives some technical results about cofibrant symmetric spectra of
  simplicial sets, and uses them to produce a model of the $K$-theory of a
  Waldhausen category that lands in the subcategory of (projective) cofibrant
  spectra.  
\end{abstract}


\section*{Introduction}

When working with algebraic $K$-theory there is often a tension of approaches.
On one hand, working with $\infty$-categories and abstract approaches produces
satisfying and beautiful results, such as Blumberg--Gepner--Tabuada's result
about the universal property of $K$-theory \cite{blumberggepnertabuada}.  In
this approach it is not important to construct models, and this often
strengthens and streamlines results.  Unfortunately, sometimes (such as in cases
where there is not a stable $\infty$-categorical structure, such as
combinatorial $K$-theory \cite{zakharevich_assembler,BGMMZ}) it is necessary to work
directly with explicit constructions of $K$-theory, such as Waldhausen's
$S_\dotp$-construction.  Unfortunately, with such direct constructions we often
run into problems: we must work with model structures, and thus we often need to
know that the spectra that we are working with are cofibrant.

Waldhausen's $S_\dotp$-construction naturally produces a symmetric spectrum of
simplicial sets.  Symmetric spectra of simplicial sets have three main model
structures on them: the level model structure, the flat model structure, and the
projective model structure.  These have decreasing collections of cofibrations,
with the projective model structure being the one most useful for many
applications.  
It is thus desirable to know when the $S_\dotp$-construction
produces projectively cofibrant spectra.

\begin{maintheorem}[Proposition~\ref{prop:Kflat}, Corollary~\ref{ex:K_not_projective}]
  For a Waldhausen category $\C$, the $S_\dotp$-construction produces a
  symmetric spectrum of simplicial sets $K(\C)$ which is flat but not projective
  cofibrant.  
\end{maintheorem}

This is unfortunate, but it turns out that it is possible to fix it.

\begin{maintheorem}[Definition~\ref{def:cofibK}, Corollary~\ref{cor:K'proj}]
  There exists a functor $K': \WaldCat \rto \Sp$ together with a natural weak
  equivalence $K' \Rto K$ such that $K'(\C)$ is projective cofibrant for all
  Waldhausen categories $\C$.
\end{maintheorem}

Moreover, both of these functors behave well on ``injective'' functors:
\begin{maintheorem}[Proposition~\ref{prop:injflat}, Corollary~\ref{cor:injproj}]
  Let $F: \C \rto \D$ be an exact functor between Waldhausen categories which is
  faithful and injective on objects.  Then $K(F)$ is a flat cofibration and
  $K'(F)$ is a projective cofibration.
\end{maintheorem}

We prove these results by developing an explicit description of flat cofibrant
spectra (Proposition~\ref{prop:flat->inc}).
With this explicit description we strengthen the cofibrancy results to simplicial spectra, as
well.  This is useful in particular when taking realizations of symmetric
spectra: since a realization is a homotopy colimit, knowing that the spectrum is
Reedy projective cofibrant implies that the realization is projective cofibrant.

\begin{maintheorem}[Proposition~\ref{prop:reedycofib}]
  For a simplicial Waldhausen category $\C_\dotp$, the simplicial spectrum
  $K(\C_\dotp)$ is Reedy cofibrant in the flat model structure, and
  $K'(\C_\dotp)$ is Reedy cofibrant in the projective model structure.
  Consequently, their geometric realizations are cofibrant in the flat and
  projective model structures, respectively.
\end{maintheorem}

\subsection*{Organization}

Section~\ref{sec_tech_foundations} introduces the concept of \emph{latching} and
\emph{increasing spectra}, which are the foundational concepts for proving the
main theorems.  It also states the main technical results about increasing
spectra that are needed to prove the main theorems, including the foundational
claim that increasing spectra are exactly flat spectra.
Section~\ref{sec:cofib_K_theory} uses the results of
Section~\ref{sec_tech_foundations} to prove the first three main theorems about
$K$-theory of Waldhausen categories; Section~\ref{sec:k_simplicial_waldhausen}
discusses the application to simplicial Waldhausen categories.
Section~\ref{sec:inc->flat} proves that a spectrum is increasing if and only if
it is flat, modulo some
technical combinatorial results that are proved in
Section~\ref{sec:tech:inc->flat}.  

\subsection*{Notation and conventions}

In this paper we work exclusively with symmetric spectra of simplicial sets; we do not specify ``of
simplicial sets'' anywhere beyond this section.

$\mathbb{S}$ denotes the sphere spectrum $S^0,S^1,S^1\sma S^1,S^1\sma S^1\sma
S^1,\ldots$, with identity morphisms as the structure maps.

The symmetric group is denoted $\Sigma_n$.

\subsection*{Acknowledgements}

Zakharevich was supported in part by NSF grant DMS-2405035.  Ponto was partially supported by NSF grant DMS-2404503.

\section{Foundational technical results}\label{sec_tech_foundations}
In this section we state the essential technical results of the paper.   As they
are easy to state and apply, but laborious to prove, we start with the
application and postpone the proofs to Sections \ref{sec:inc->flat} and \ref{sec:tech:inc->flat}.

 The spectrum $\redS$, the \emph{truncated
    sphere spectrum}, is defined to be the spectrum which has a single point in
  degree $0$ and then $S^q$ in degree $q$ for $q>0$. 
\begin{definition}[{\cite[Definition 5.2.1]{hoveyshipleysmith},
    \cite[Construction I.5.29, Proposition I.5.39]{schwedebook}}]\label{defn:traditional_latching}
  Let $X$ be a symmetric spectrum.  The \emph{traditional
    latching space} $L_n^TX$, is the $n$-level of $X \sma \redS$.
\end{definition}

The map $\nu^T:X \sma \redS \rto  X\sma \mathbb{S}$ has at level $n$ a
map 
\begin{equation}\label{eq:nu}\nu_n^T:L_n^TX \rto X_n.\end{equation}

\begin{definition}[{\cite[Definition III.3.1]{schwedebook}}] \label{def:flat} An
   spectrum $X$ is \emph{flat} (i.e. cofibrant in the stable flat
  model structure) if for all $n \geq 0$ the map $\nu_n^T$ is a cofibration.  It
  is \emph{projective} (i.e. cofibrant in the stable projective model structure)
  if for all $n \geq 0$ the map $\nu_n^T$ is a free $\Sigma_n$-cofibration.  
  
  A
  morphism $X \rto Y$ of spectra is a \emph{flat cofibration}
  (resp. \emph{projective cofibration}) if for all $n \geq 0$ the induced maps
  $L_n^TY \cup_{L_n^T X} X_n \rto Y_n$ are cofibrations (resp. $\Sigma_n$-free
  cofibrations).
\end{definition}

\begin{remark}\label{rmk:characterize_cofib}
For the purpose of this paper, it is sufficient to know that cofibrations of
$\Sigma_n$-simplicial sets are exactly the monomorphisms (this is stated in
\cite[Below Definition III.4.1]{schwedebook}), and that the free
$\Sigma_n$-cofibrations are those cofibrations where the $\Sigma_n$-action is
free on the complement of the image (see \cite[Proposition V.2.4, Corollary
V.2.10]{goerssjardine}, applied to the case where $G_\dotp$ is constant).
\end{remark}

Before beginning to analyze the structure of spectra, we will need several
facts about shuffles.

\begin{definition}
  A permutation $\tau\in \Sigma_n$ is an \emph{$(n-k,k)$-shuffle} if $\tau(1) <
  \cdots < \tau(n-k)$ and $\tau(n-k+1) < \cdots < \tau(n)$.
\end{definition}

The following are the basic facts about shuffles that we will be using:
\begin{lemma}  \label{lem:d_latching}
  \begin{enumerate}
  \item The $(n-k,k)$-shuffles represent the cosets of
    $\Sigma_{n-k}\times \Sigma_k$ in $\Sigma_n$.  In other words, for any
    $\gamma\in \Sigma_n$ there exists a unique factorization
    $\gamma = \gamma'(\tau,\rho)$, where $\gamma'$ is an $(n-k,k)$-shuffle and
    $(\tau,\rho)\in \Sigma_{n-k}\times \Sigma_k$.
  \item For a $(k,n-k)$-shuffle $\gamma$ and an $(\ell, k-\ell)$-shuffle $\tau$,
    there is a unique $(\ell, n-\ell)$-shuffle $\gamma'$ and a unique
    $\rho \in \Sigma_{n-\ell}$ such that $\gamma(\tau,1) = \gamma'(1, \rho)$.
  \end{enumerate}
\end{lemma}

We are now ready to define \emph{latched simplices}.

\begin{definition}\label{def:latch}
  Let $X$ be a symmetric spectrum.  Fix $n \geq 1$ and let $0<k<n$. Let
  $\delta\in S^k$ be a simplex and $\tau$ a $(n-k,k)$-shuffle.  We say that a
  simplex \emph{$y\in X_{n-k}$ is $(\tau,\delta)$-latched to $x\in X_n$} if
\[\tau\cdot \sigma_k(y,\delta) = x\]
 for
  $\sigma_k: X_{n-k} \sma S^k \rto X_n$ the structure map.  
  
    When $\tau$ and
  $\delta$ are clear from context (or unimportant) we simply say that \emph{$x$
    is latched to $y$}.
  Further, we say a simplex $x$ is \emph{latched} if there is a tuple $(k,y,\delta,\tau)$ so that $y\in X_k$ 
 is $(\delta,\tau)$-latched to $x$, otherwise we say that $x$ is
 \emph{unlatched}.

   Let $\Lambda_nX$ be the full simplicial subset of latched simplices. %
 \end{definition}

  \begin{notation}
  We will write $\sigma$ instead of $\sigma_k$ for structure
  maps when $k$ is clear from context.  Thus for a simplex $\delta\in S^k$, we
  write $\sigma(x,\delta)$ to mean $\sigma_k(x,\delta)$.
\end{notation} 

\begin{lemma}\label{lem:Lambda_well_defined}
  $\Lambda_nX$ is a well-defined $\Sigma_n$-simplicial subset of $X_n$.
\end{lemma}

\begin{proof}
By direct computation $\Lambda_nX$ is a well-defined simplicial subset, because
if $x$ is $(\delta,\tau)$-latched to $y\in X_{k-1}$ then $d_ix$ is
$(d_i\delta,\tau)$-latched to $d_iy$ and $s_ix$ is $(s_i\delta,\tau)$-latched to
$s_iy$.  

  Now we show that the $\Sigma_n$-action is well-defined.
    We must show that if $x\in \Lambda_n X$ and $\rho\in \Sigma_n$ then
  $\rho\cdot x \in \Lambda_n X$.  Suppose that $x$ is $(\tau,\delta)$-latched to
  $y$, so that $\tau\cdot \sigma_k(y,\delta) = x$.  Then
  $(\rho\tau) \cdot \sigma_k(y,\delta) = \rho\cdot x$.  Since the shuffles
  represent the left cosets of $\Sigma_{n-k} \times \Sigma_k$ in $\Sigma_n$,
  there exists a unique shuffle $\tau'$ and
  $(\rho_{n-k}, \rho_k)\in\Sigma_{n-k}\times \Sigma_k$ such that
  $\rho\tau = \tau'(\rho_{n-k},\rho_k)$.  Then
  \[\rho\cdot x = \tau' \sigma_k(\rho_{n-k}y, \rho_k\delta).\]
  Thus $\rho x$ is $(\tau',\rho_k\delta)$-latched to $\rho_{n-k} y$, and so
  $\rho x\in \Lambda_n X$, as desired.
\end{proof}

The point of $\Lambda_n$ is that it is a more hands-on definition of the image
of the traditional latching object.

\begin{lemma} \label{lem:nuL=lambda}
  For any spectrum $X$, $\nu^T(L^T_nX) = \Lambda_nX$.
\end{lemma}

\begin{proof}
  Suppose that $x\in \Lambda_nX$.  Then there exists an $n-1$-simplex $y$ such
  that $x$ is $(\delta,\tau)$-latched to $y$.  But then $x$ is in the image of
  $(\tau,y,\delta)\in L^T_nX$, and thus is in the image of $L^T_nX$; thus
  $\Lambda_nX \subseteq \nu^T(L^T_nX)$.  On the other
  hand, consider a simplex represented by $(\tau,y,\delta)\in L^T_nX$. Then
  $\nu^T(\tau,y,\delta) = \tau\cdot \sigma(y,\delta)$, and thus its image in
  $X_n$ is in $\Lambda_n X$.   Thus $\nu^T(L^T_nX) \subseteq \Lambda_nX$, and
  the two are equal, as desired.
\end{proof}

\begin{corollary}\label{cor:incflatmor} Suppose that $Y$ is
  flat.  Then a level cofibration $f:X \rto Y$ is a flat cofibration if and only
  if for all $n$,
\[f_n(X_n)\cap \Lambda_n Y\subseteq f_n(\Lambda_n X).\]
\end{corollary}

\label{page_ref:incflatmor}
\begin{proof}
  Our goal is to show that $f$ is a flat cofibration---in other words, that the
  induced map
  \[\nu_n: L^T_n Y \cup_{L^T_n X} X_n \rto Y_n\]
  is a cofibration (i.e. is monic)---if and only if $f_n(X_n) \cap \Lambda_nY
  \subseteq f_n(\Lambda_n X)$.
  
  First note that for a pushout diagram of simplicial sets
  \begin{diagram}
    { C & B \\ A & A \cup_C B \\ & & D \\};
    \to{1-1}{1-2}^\ell \to{1-1}{2-1}_l \to{1-2}{2-2} \to{2-1}{2-2}
    \diagArrow{->,bend left}{1-2}{3-3}^g \diagArrow{->, bend right}{2-1}{3-3}_f
    \to{2-2}{3-3}!h
  \end{diagram}
  if $f$ and $g$ are monic then $h$ is monic if and only if $f(A)\cap g(B)\subset 
  f(k(C))=g(\ell(C))$.  Abusing notation we write this as $h(A)\cap h(B)\subset
  h(C)$.

  In our case, this is the diagram:
  \begin{diagram}
    { L^T_nX & L^T_nY \\ X_n & L^T_nY\cup_{L^T_nX} X_n \\ & & Y_n \\};
    \to{1-1}{1-2} \to{1-1}{2-1}_{\nu^T_n} \to{1-2}{2-2} \to{2-1}{2-2}
    \diagArrow{->,bend left}{1-2}{3-3}^{\nu^T_n} \diagArrow{->, bend right}{2-1}{3-3}_{f_n}
    \to{2-2}{3-3}!{\nu}
  \end{diagram}
  The map $f_n$ is monic by assumption, since $f$ is a level cofibration.  The
  map $\nu_n^T$ for $Y$ is monic, since $Y$ is flat.  Thus $\nu$ is monic if and
  only if
  \[f_n(X_n) \cap \nu_n^T(L_n^T Y) \subseteq f_n(\nu_n^T(L_n^T X)).\] But by
  Lemma~\ref{lem:nuL=lambda}, for any spectrum $Z$,
  $\nu_n^T(L_n^T(Z)) = \Lambda_nZ$.  Thus $\nu$ is monic if and only if
  \[f_n(X_n) \cap \Lambda_nY \subseteq f_n(\Lambda_n X),\]
  as desired.
\end{proof}

In many cases of interest (in particular, the $K$-theoretic examples we are
working with), the structure maps in $X$ are all monic, with symmetric actions
being generated somewhat ``freely,'' by which we mean that structure maps are
monic (and thus do not lose information) and the symmetric group action does not
``repeat information''.  We encode this in the definition of an \emph{increasing
  spectrum}.

\begin{definition} \label{def:increasing}
  A symmetric spectrum  is \emph{increasing} if for all $n \geq 1$, 
  \begin{enumerate}
  \item \label{it:structure->monic} the map $\sigma: X_{n-1}\sma S^1 \rto X_n$ is monic, and
  \item \label{it:latched} for any $\rho\in \Sigma_n$ such that $\rho(n) \neq
    n$, and every $y_1,y_2\in X_{n-1}\sma S^1$, if $\sigma(y_1) = \rho\cdot
    \sigma(y_2)$ then there exists a simplex $z\in X_{n-2}$ such that both $y_1$
    and $y_2$ are latched to $z$.
  \end{enumerate}
\end{definition}

Then the main technical result of this paper is the following: 

\begin{theorem} \label{thm:flat<->inc}
  A spectrum is flat if and only if it is increasing.
\end{theorem}

\begin{proof}
  The forwards
  implication is Proposition~\ref{prop:flat->inc}.
  The backwards implication is Proposition~\ref{prop:inc->flat}.  
  \end{proof}

We now consider projective spectra. 

\begin{theorem}[See Page~\pageref{proof:characterize_flat_projective_spectra}] \label{thm:characterize_flat_projective_spectra}
  Let $X$ be a flat spectrum.  Then $X$ is projective if and only if  $X_n\setminus 
\Lambda_nX
$ is $\Sigma_n$-free for all $n$. 
\end{theorem}

The corresponding characterization of morphisms is the following:
\begin{corollary}[See Page~\pageref{page_ref:flatY->proj}] \label{cor:flatY->proj}
  Suppose that $Y$ is projective.  Then a flat cofibration $f:X \rto Y$ is a
  projective cofibration.
\end{corollary}

\section{Cofibrant $K$-theory}\label{sec:cofib_K_theory}

The goal of this section is to give a construction of $K(\C)$ for a Waldhausen
category $\C$ which will give a symmetric spectrum which is cofibrant in the
projective stable model structure on symmetric spectra. 

In this section, $\C$ will always be a Waldhausen category.  We will refer to
the ``all-at-once'' $K$-theory construction of \cite[Construction
2.2]{blumbergmandell11} as the \emph{usual $S_\dotp$-construction}.  This
construction produces a symmetric spectrum.  We use the notation of
\cite[Construction 2.2]{blumbergmandell11}, except that we simplify their
subscripts $n_1,\ldots,n_q$ as $\vec n$.  In order for this construction to be
well-defined as a symmetric spectrum it is necessary to have a distinguished
zero object, and thus many people assume that there is a unique zero object in a
Waldhausen category.  We instead make the assumption that $\C$ is
\emph{pointed}: that there is a \emph{chosen} zero object that is distinguished,
but that other zero objects may exist.

\begin{definition}
  A diagram $X:\Ar[n] \rto \C$ is \emph{single-object} if there exists an object
  $A\in \C$ such that for all objects $x\in
  \Ar[n]$, if $X(x) \neq 0$ then $X(x) = A$, and for all morphisms $\alpha\in
  \Ar[n]$, if $X(\alpha)$ is not an identity morphism then it is equal to the
  morphism $0 \rcofib A$.

  A diagram $A$ in $S^{(q)}_{n_1,\ldots,n_q}\C$ is \emph{$i$-trivial} (for
  $1 \leq i \leq q$) if the diagram $A$ is single-object when
  considered as a functor \[\Ar[n_i] \rto
  S^{(q-1)}_{n_1,\ldots,\widehat{n_i},\ldots,n_q}\C.\]
  \end{definition}

\begin{example}\label{ex:i_trivial}
 Consider, the
following $2$-simplex in $S_\dotp S_\dotp \C$ associated to a non-isomorphism
cofibration $A \rcofib B$ (where $A \neq 0$):
\begin{diagram}
  { A & A \\ A & B \\};
  \cofib{1-1}{1-2} \cofib{2-1}{2-2}
  \cofib{1-1}{2-1} \cofib{1-2}{2-2}
\end{diagram}
This  is not
$1$-trivial or $2$-trivial. If $A=0$ then this would be both $1$-trivial and
$2$-trivial.  The diagram
\begin{diagram}
  { A & B \\ A & B \\};
  \cofib{1-1}{1-2} \cofib{2-1}{2-2} \eq{1-1}{2-1} \eq{1-2}{2-2}
\end{diagram}
is $2$-trivial but not $1$-trivial.
\end{example}

In the context of Waldhausen's $S_\dotp$-construction, $\Lambda_n K(\C)$ can be
described in a straightforward fashion.
 
\begin{lemma} \label{lem:inLambda}
  For all $n \geq 0$, $\Lambda_n K(\C) \subseteq K(\C)_n$ is the subspace of
  those diagrams $A$ for which there exists some $1\leq i \leq n$ so that $A$ is
  $i$-trivial. 
\end{lemma}

\begin{proof}
  In the usual construction for $K(\C)$, the structure map $\sigma$ takes a
  diagram of dimension $n-1$ and creates a diagram of dimension $n$ which is
  $n$-trivial.  (See, for example, \cite[Section 6]{Z-closed} for a discussion
  of this construction in detail.)  The $\Sigma_n$-action on diagrams
  $\Ar[\vec m]$ permutes the coordinates of $\Ar[\vec m]$; so if a
  diagram $A$ is $i$-trivial for all $i\in I$ (for some
  $I \subseteq \{1,\ldots,q\}$) then the diagram $\tau\cdot A$ will be
  $\tau(i)$-trivial for all $i\in I$.  In particular, all latched simplices must
  be $i$-trivial for some $i$.
\end{proof}

Using this characterization we can show that $K(\C)$ is flat.

\begin{proposition} \label{prop:Kflat}
  For the usual $S_\dotp$-construction,   $K(\C)$ is increasing, and therefore
  flat.
\end{proposition}

\begin{proof}
  Since all of the data of a diagram is preserved by $\sigma_1$, $\sigma_1$ is
  monic and thus all structure maps of $K(\C)$ are monic.  Moreover, if $\rho$
  has $\rho(n) \neq n$ and $\sigma(x,\delta) = \rho\cdot \sigma(x',\delta')$
  then this implies that the diagram represented by $x$ was $n$-trivial and
  $\rho(n)$-trivial.  In particular, if we take the subdiagram in which the
  $n$-th and $\rho(n)$-th coordinates are both set to the maximum value, we get
  an $n-2$-simplex that both $x$ and $x'$ are latched to.
 Thus $K(\C)$ is increasing, and flat by
  Theorem~\ref{thm:flat<->inc}.
\end{proof}

\begin{corollary}  \label{ex:K_not_projective}
  $K(\C)$ is not a projective spectrum if there exists a non-isomorphism
  cofibration $A \rcofib B$ with $A \neq 0$.
\end{corollary}

\begin{proof}
  The 2-simplex in Example~\ref{ex:i_trivial} is fixed under the nontrivial
  element in $\Sigma_2$ (since the symmetric action just permutes the axes).  It
  is also not in $\Lambda_2K(\C)$, since it is not $1$-trivial or $2$-trivial.  So
  this is a simplex in $K(\C)_2$ with a nontrivial stabilizer, but it is not in
  $\Lambda_2K(\C)$; thus the usual $S_\dotp$-construction does not give a projective
  symmetric spectrum.  
\end{proof}

Although the following result is not needed for the purposes of this paper, we
include it as a technical result which may be useful in the future.

\begin{proposition} \label{prop:injflat}
  Let $F:\C \rto \D$ be an exact faithful functor of Waldhausen categories which
  is injective on objects.  Then with the usual $S_\dotp$-construction,
  $K(\C) \rto K(\D)$ is a flat cofibration.
\end{proposition}

\begin{proof}
  The conditions on the functor imply that $K(\C) \rto K(\D)$ is a level
  cofibration. 
  By
  Corollary~\ref{cor:incflatmor}, 
  it suffices to check that
  \[K(F)_n(K(\C)_n) \cap \Lambda_n K(\D) \subseteq K(F)_n(\Lambda_nK(\C)).\]
   By Lemma~\ref{lem:inLambda}, the
  space $\Lambda_nK(\D)$ is the simplicial set of $n$-dimensional diagrams in
  $\D$ which are $i$-trivial in some direction $i$.  The space $K(F)_n(K(\C)_n)$
  is the simplicial set of $n$-dimensional diagrams in $\D$ which are contained
  in the image of $\C$.  The intersection of these is exactly those diagrams in
  the image of $\C$ which are $i$-trivial in some direction $i$---which is
  exactly $K(F)_n(\Lambda_n K(\C))$.  Thus the map is a cofibration, as desired.
\end{proof}

\begin{definition}
  We write $\WaldCat$ for the category of Waldhausen categories with a
  distinguished zero object, i.e. for Waldhausen categories where one of the
  isomorphic copies of the zero object is chosen as a part of the
  data.\footnote{Waldhausen \cite{waldhausen} originally defined his categories
    to be pointed in this way; in other sources, including \cite{weibel_kbook}, it is
    only assumed that there is a zero object without an emphasis on this being a
    part of the data.} The distinguished zero object is denoted $0$.
    
      When defining the $S_\dotp$-construction,
  we enforce that all objects which are forced to be $0$ by the definition of
  the construction are the distinguished zero object.
 \end{definition}   
    Let
  $A: \Ar[\vec n] \rto \C$ be a functor.  The \emph{zero subsieve} of
  $A$ is the largest sieve in $\Ar[\vec n]$ such that for every object
  $c$ in the sieve, $A(c)=0$.  
Any object of the usual $S_\dotp$-construction has a nonempty zero subsieve by
definition.  

For our construction of cofibrant $K$-theory we require that each object in the category has 
infinitely many objects that are isomorphic.  If necessary we replace the category with an equivalent one as in the following definition.

\begin{definition} \label{def:cofibK} Let $\C$ be a Waldhausen category with a
  distinguished zero object, and let $\tilde\Z$ be the category whose objects
  are the integers and which has a unique isomorphism between any two
  objects. 
  
  Let $\C_{\tilde \Z} = \C \times \Z$, with $(0,0)$ the distinguished zero
  object.  All the objects of the form $(0,n)$ for $n \neq 0$ are called
  \emph{undistinguished zeroes}.  Write $u_\Z: \C_{\tilde \Z} \rto \tilde\Z$ for
  the forgetful functor.
\end{definition}  

In order to produce a model of $K$-theory where fixed points such as the one in
Example~\ref{ex:i_trivial} do not occur, we wish to restrict attention to those
diagrams with all entries distinct.  The point of considering $\C_{\tilde \Z}$ instead of
$\C$ is to ensure, in a functorial manner, that there exist enough copies of
each object to make this possible, and to make sure that even when a functor is
not injective on objects, it will take diagrams with nonrepeating entries to
diagrams with nonrepeating entries.

\begin{definition}
  We say that  an
  object $X\in S^{(q)}_{\vec n}\C_{\tilde\Z}$ is
  \emph{nonrepeating} if the function $u_\Z X: \ob \Ar[\vec n]
  \rto \Z$,
  \begin{itemize}
  \item  is injective away from the preimage of $0$, and
  \item the preimage of $0$ is the zero subsieve of $X$.
  \end{itemize}
  We say that $X$ is \emph{essentially nonrepeating} if $X = sY$ where $s$ is
  some composition of degeneracies and $Y$ is nonrepeating.

  Let $\tilde S^{(q)}_{\vec n} \C_{\tilde \Z}$ be the full Waldhausen subcategory of
  $S^{(q)}_{\vec n}(\C_{\tilde \Z})$ containing only essentially nonrepeating
  diagrams.
\end{definition}
  
Since repetitions in an essentially nonrepeating diagram can be obtained only
through the application of degeneracies, it can be shown by induction that the
only way for an object to repeat in an essentially nonrepeating diagram is in a
contiguous rectangle.
  
  \begin{proposition} \label{prop:Walequiv}
  The forgetful functor $\tilde S^{(q)}_{\vec n} \C_{\tilde \Z}\to S^{(q)}_{\vec n}\C$ is an
  equivalence of Waldhausen categories.
  \end{proposition}
  
  \begin{proof}
    Since $\tilde \Z$ has infinitely many isomorphic objects, every object in
    $S^{(q)}_{\vec n}\C$ is isomorphic to an object with no duplicated
    $\tilde \Z$-coordinates.
  \end{proof}
   
  \begin{remark}
    The undistinguished zeroes are necessary for this proposition to hold,
    because if $\C$ contains a nonzero object $A$ with a cofibration
    $A \rcofib 0$, then the object $0\rcofib A \rcofib 0$ in $S_2\C$ would
    \emph{not} have a preimage in $\tilde S_2\C_\Z$ if nondistinguished zeroes
    did not exist.
  \end{remark}
  
  \begin{definition}

  Define
  \[K'(\C)_q = |w\tilde S^{(q)}_\dotp \C_{\tilde \Z}|.\]
  We call this the \emph{cofibrant construction of $K$-theory}.
\end{definition}

The forgetful functor induces a map $\iota: K'(\C) \rto K(\C)$.
\begin{corollary}
  The map $\iota$ is a weak equivalence.
\end{corollary}

\begin{proof}
  Proposition~\ref{prop:Walequiv} implies that $\iota$ is a level equivalence,
  which implies that it is a weak equivalence.
\end{proof}

As a first step to showing that $K'(\C)$ is projective we will show that it is
flat.  The following is proved identically to Proposition~\ref{prop:Kflat}.

\begin{lemma} \label{lem:K'flat}
  For the cofibrant construction of $K$-theory,  $K'(\C)$ is
  increasing and therefore flat.
\end{lemma}

\begin{remark}
  Exactly as observed in Lemma~\ref{lem:inLambda}, for $k \geq 1$, the
  space $\Lambda_kK'(\C)$ is the subspace of $K'(\C)_k$ which contains those
  diagrams which are $i$-trivial for some $1\leq i \leq k$.
\end{remark}

We turn our attention to proving that $K'(\C)$ is projective.  To prove this we
need a preliminary result showing that the simplices which make $K(\C)$ fail to
be projective are not present inside $K'(\C)$.  

\begin{proposition} \label{prop:fix->triv} Let $A\in \tilde S^{(q)}_{\vec
    n}\C_{\tilde \Z}$
  and suppose that $\tau \cdot A = A$ for some nonidentity
  $\tau \in \Sigma_q$.  Then $A$ is $i$-trivial in some direction $i$.
\end{proposition}

\begin{proof}
  Since the action of $\tau$ permutes the $S_\dotp$-directions, if
  $A = \tau\cdot A$ and the cycle decomposition $\tau$ contains a $k$-cycle
  $\tau'$ then $A = \tau' \cdot A$.  If $\tau' = (a_1\ \cdots\ a_k)$ then
  $(a_1\ a_2) \cdot A = A$.  An object $A$ is $i$-trivial if and only if
  $\sigma\cdot A$ is $\sigma(i)$-trivial; in particular it suffices to prove the
  statement of the proposition in the case when $\sigma = (1\ 2)$.  Moreover,
  since
  \[S^{(q)}_{n_1,\ldots,n_q}S^{(r)}_{n_{q+1},\ldots,n_{q+r}}\C \cong
    S^{(q+r)}_{n_1,\ldots,n_{q+r}}\C\]
  proving the statement for the case of $\sigma$ is acting on $S_{n,m}^{(2)}\C$
  is sufficient (by replacing $\C$ with $S^{(q-2)}_{n_3,\ldots,n_q}\C$ if
  necessary).

  Suppose that $A\in S^{(2)}_{m,n}\C$ is fixed by $(1\ 2)$.  In this proof we
  will focus on objects of the form $(0<i,0<j)$ in $\Ar[m,n]$, writing them (by
  an abuse of notation) simply as $(i,j)$.  The action of $\sigma$ takes
  $A(i,j)$ to $A(j,i)$, so we must have $m=n$.  Let $i$ be the minimal integer
  such that $(i,n)$ is outside of the zero subsieve of $A$; let $B = A(i,n)$.
  Since $A$ is fixed by $\sigma$, $A(n,i) = B$, as well; thus $A(i,i)=A(n,n)=B$,
  as well, since $A$ is essentially nonrepeating.   Note that every
  $(a,b)\in \Ar[n,n]$, if $a < i$ then it is in the zero subsieve, and thus (by
  the $\sigma$-invariance of $A$) if $b < i$ it is also in the zero subsieve.
  Thus we see that $A(a,b)$ is either equal to $B$ or it is the zero object.  In
  other words, this is a diagram that contains only two types of morphisms:
  identity morphisms, and the morphism $0 \rcofib B$, in both direction; in
  other words, it is both $1$-trivial and $2$-trivial.
\end{proof}

\begin{corollary} \label{cor:K'proj}
  $K'(\C)$ is projective.
\end{corollary}

\begin{proof}
  Since $K'(\C)$ is flat, in order to check that $K'(\C)$ is projective by Theorem~\ref{thm:characterize_flat_projective_spectra} it
  remains to show that the $\Sigma_n$-action on $K'(\C)_n \backslash \Lambda_nK'(\C)$
  is free. 
  
  Let $\tau\in \Sigma_n$ and suppose that $\tau \cdot A = A$ for
  some $A\in K'(\C)_n$ and nonidentity $\tau$.  Then by
  Proposition~\ref{prop:fix->triv}, $A$ is $i$-trivial in some direction $i$, so
  by Lemma~\ref{lem:inLambda}, $A\in \Lambda_nK(\C_\Z)$.  Since $A$ is
  essentially nonrepeating, $A\in \Lambda_n K'(\C)$, and thus $\Sigma_n$ acts freely
  only $K'(\C)_n \backslash \Lambda_nK'(\C)$, as desired.
\end{proof}

\begin{remark}
  If $\C$ has infinitely many objects in each nonzero isomorphism class, and if
  there do not exist cofibrations $A \rcofib 0$ where $A$ is nonzero, then we
  could define $K'(\C)$ to be the subspectrum of $K(\C)$ containing only
  essentially nonrepeating diagrams.  The same proof as above works to prove
  that this model is also projective cofibrant.  However, the construction will
  no longer be functorial.
\end{remark}

\begin{corollary} \label{cor:injproj}
  Let $F: \C \rto \D$ be a faithful exact functor which is injective on
  objects.  Then $K'(F)$ is a projective cofibration.  
\end{corollary}

\begin{proof}
  By Corollary~\ref{cor:K'proj} $K'(\D)$ is projective; by definition $K(F)$ is
  a level cofibration.  We now show that $K'(F)$ is flat.
  Following the same lines as the proof of Proposition~\ref{prop:injflat}, we see
  that it suffices to show that
  \[K'(F)_n(K'(\C)_n) \cap \Lambda_n K'(\D) \subseteq K'(F)_n(\Lambda_n
    K'(\C)).\]
  By the same logic as in that proof, this is true.  Thus $K'(F)$ is a flat
  cofibration; by Corollary~\ref{cor:flatY->proj}, 
  it is also a projective cofibration.
\end{proof}

\section{The $K$-theory of simplicial Waldhausen categories}\label{sec:k_simplicial_waldhausen}

The goal of this section is to show that $K'$ behaves well with respect
to the geometric realization of simplicial spectra.  In particular, our goal is
to prove the following theorem:
\begin{theorem} \label{thm:reedycofib} Let $\C_\dotp$ be a simplicial object in
  the category of Waldhausen categories.  Then, in the projective (resp. flat)
  model structure on symmetric spectra, $|K'(\C_\dotp)|$ (resp. $|K(\C_\dotp)|$)
  is cofibrant.
\end{theorem}

We begin with a review of Reedy cofibrancy of simplicial objects in a model
category. 

\begin{definition}[{\cite[Definition 15.2.5]{hirschhornbook}}]
   Let $I_n$ be the category whose objects are morphisms
  $[n] \rto [j]$ in $\Delta$ with $j < n$, with morphisms being given by
  commutative triangles.
  
  Let $\C$ be a cocomplete category, and let $X_\dotp\in s\C$ be a simplicial
  object in $\C$.   Then $X_\dotp$ induces a functor
  $\tilde X: I_n^\op \rto \C$ by $\tilde X([n] \rto [j]) = X_j$.  Define
  \emph{$n$-th Reedy-latching object} of $X_\dotp$, denoted $L^\R_nX_\dotp$, by
  \[L_n^\R X_\dotp \defeq \colim_{I_n^\op} \tilde X.\]
  There is a natural transformation $\tilde X \rto 1_{X_n}$ which induces a
  morphism $L_n^\R X_\dotp \rto X_n$.  
\end{definition}

\begin{example} \label{ex:latchbisimp}
  When $\C = \Set$ and $X_\dotp$ is a simplicial set, $L_n^\R X_\dotp$ is the
  subset of $X_n$ given by the degenerate simplices.  If $\C = s\Set$ and
  $X_\dotp$ is a bisimplicial set, then $L_n^\R X_\dotp$ is the subspace of
  $X_n$ given by the union of the images of $s_iX_{n-1}$ for $0 \leq i \leq
  n-1$.
  
  The reason that we can just take the union of the images is that all of the
  maps $s_i$ have left inverses, and intersect in predictable ways.  In
  particular, suppose that $s_iy = s_jz$ for some simplices $y$ and $z$.  Then
  if $i =j$ we must have $y=z$, by applying $d_i$ on the left.  If $i < j$ then
  we must have $y = s_{j-1}d_iz$ and $z = s_id_iz$.  Thus in the colimit
  defining $L_n^\R X_\dotp$, simplices whose image are equal in $X_n$ are
  already equal.
\end{example}

\begin{example} \label{ex:latchsp}
  Let $\C = \Sp$.  Then $L_n^\R X_\dotp$ is the subspectrum of $X_n$ which, at
  level $k$ is the union of the images of the spaces $s_i(X_{n-1})_k$.  (This
  uses Example~\ref{ex:latchbisimp} and the fact that colimits in $\Sp$ are
  levelwise.)
\end{example}

\begin{remark}
  The standard term for the Reedy latching object is simply ``latching object'',
  and it is denoted simply $L_n$; as this would lead to ambiguity in our
  notation, we have renamed it.
\end{remark}

\begin{definition}
  Let $\C$ be a model category. A morphism $X_\dotp \rto Y_\dotp$ in $s\C$ is a
  \emph{Reedy cofibration} if for all $n$ the induced morphism
  \[\kappa_n: L_n^\R Y_\dotp \cup_{L_n^\R X_{\dotp}} X_n \rto Y_n\]
  is a cofibration in $\C$.  When $X_\dotp$ is the initial object of $s\C$ then the
  unique morphism $\initial \rto Y_\dotp$ is a Reedy cofibration (i.e. $Y_\dotp$
  is \emph{Reedy cofibrant}) if for all $n$ the morphism $L_n^\R Y \rto Y_n$ is
  a cofibration in $\C$.

  When $\C= \Sp$ we will say that $X_\dotp \rto Y_\dotp$ is a \emph{flat Reedy
    cofibration} (resp. \emph{projective Reedy cofibration}) if for all $n$ the
  map $\kappa_n$ is a flat (resp. projective) cofibration.
\end{definition}

\begin{example}
  By the calculation in Example~\ref{ex:latchbisimp}, all bisimplicial sets are
  Reedy cofibrant, and more generally all pointwise monomorphisms of
  bisimplicial sets are Reedy cofibrations. (See \cite[Theorem
  IV.3.9]{goerssjardine} for an in-depth discussion of this example.)
\end{example}

The main goal of this section is to analyze the Reedy latching objects of the
$K$-theory of simplicial Waldhausen categories and thereby prove that the
$K$-theory of their homotopy colimits are cofibrant.  We begin with the
following characterization of the Reedy latching object.

The main technical result needed is the following:
\begin{proposition} \label{prop:reedycofib} Let $\C_\dotp$ be a simplicial
  object in the category of Waldhausen categories.  Then $K(\C_\dotp)$ is flat
  Reedy cofibrant and $K'(\C_\dotp)$ is projective Reedy cofibrant.
\end{proposition}

\begin{proof}
  We present the argument for the projective case, as the proof for the flat
  case is a minor simplification of the projective case (after replacing each
  instance of $K'$ with $K$).
  
  Our goal is to prove that the map $\kappa:L_n^\R K'(\C_\dotp) \rto K'(\C_n)$ is a
  projective cofibration.  By Corollary~\ref{cor:K'proj}, $K'(\C_n)$ is
  projective.  Thus by Corollary~\ref{cor:flatY->proj} it suffices to show that
  $\kappa$ is flat.  By Example~\ref{ex:latchsp}, $\kappa$ is a level cofibration, so
  by Corollary~\ref{cor:incflatmor} it suffices to check that
  \[\kappa_m((L_n^\R K'(\C_\dotp))_m) \cap \Lambda_mK'(\C_n)
    \subseteq \kappa_m(\Lambda_m(L_n^\R K'(\C_\dotp)).\]
  In particular, as discussed in Example~\ref{ex:latchsp}, $\kappa$ is the
  inclusion of those simplices that are in the image of some $K'(s_i)$,
  $s_i: \C_{n-1} \rto \C_n$.  By Lemma~\ref{lem:inLambda}, a simplex is inside
  $\Lambda_mK'(\C_n)$ if it is $j$-trivial for some $j$. 
  Thus a simplex is in the left-hand side if it is \emph{both} $j$-trivial and
  in the image of some $K'(s_i)$.  Since $\kappa$ is a level cofibration, this
  means that a simplex in the left-hand side must be the image under some
  $K'(s_i)$ of a $j$-trivial diagram in $K'(\C_{n-1})$---i.e. it is in
  $\Lambda_m(L_n^\R K'(\C_\dotp))$. 
\end{proof}

Given this result, we are ready to prove the theorem.

\begin{proof}[Proof of Theorem~\ref{thm:reedycofib}]
  By Proposition~\ref{prop:reedycofib}, $K'(\C_\dotp)$ (resp. $K(\C_\dotp)$) is
  Reedy cofibrant for the projective (resp. flat) model structure on symmetric
  spectra.

  By \cite[Theorem III.4.11]{schwedebook}, the category of symmetric spectra of
  simplicial sets, with either the flat or the projective stable model
  structure, are simplicial.  By \cite[Proposition VII.3.6]{goerssjardine}, the
  geometric realization of simplicial objects in a simplicial model category
  preserves cofibrations between Reedy cofibrant objects.  Thus $|K'(\C_\dotp)|$
  (resp. $|K(\C_\dotp)|$) is cofibrant, as desired.
\end{proof}

Lastly, we want to mimic the results of Proposition~\ref{prop:injflat} and
Corollary~\ref{cor:injproj}, for the simplicial Waldhausen category case.

\begin{proposition}
  Let $F_\dotp: \C_\dotp \rto \D_\dotp$ be a simplicial exact functor which is
  levelwise faithful and injective on objects.  Then $K(F_\dotp)$
  (resp. $K'(F_\dotp)$) is a Reedy cofibration in the Reedy model structure
  based on the flat (resp. projective) stable model structure on symmetric
  spectra.   Consequently, its realization is a cofibration.
\end{proposition}

\begin{proof}
  We prove the projective case; the flat case is proved analogously.

  To prove that $K'(F_\dotp)$ is a Reedy cofibration we must show that the map
  \[\eta_n: L^\R_nK'(\D_\dotp) \cup_{L_n^\R K'(\C_\dotp)} K'(\C_n) \rto K'(\D_n)\]
  is a projective cofibration.  Since (by Corollary~\ref{cor:K'proj}) $K'(\D_n)$ is projective, by
  Corollary~\ref{cor:flatY->proj} it suffices to prove that this map is a flat
  cofibration. 
  
  We begin by analyzing its domain.  By Corollary~\ref{cor:injproj},
  $K(F_\dotp)$ is a level cofibration, so $K'(\C_n)$ can be considered to be a
  subspectrum of $K'(\D_n)$.  Since $K'(\D_\dotp)$ is Reedy cofibrant, as
  discussed in the proof of Proposition~\ref{prop:reedycofib},
  $L^\R_n K'(\D_\dotp)$ is also a subspectrum of $K'(\D_n)$; it is the
  subspectrum of those simplices in $K'(\D_n)$ which are in the image of
  $K'(s_i): K'(\D_{n-1} \rto \D_n)$ for some $i$.  Suppose that a simplex in
  $K'(\D_n)$ is in both $K'(\C_n)$ and $L^\R_n K'(\D_\dotp)$.  Since it is in
  $L^\R_n K'(\D_\dotp)$ it is in the image of some $K'(s_i)$; applying $K'(d_i)$
  to it gives a simplex inside $K'(\C_{n-1})$; thus the original simplex is in
  $L^\R_n K'(\C_\dotp)$.  In particular,as in the proof of Corollary~\ref{cor:incflatmor}, we see that $\eta_n$ is a level
  cofibration.

  To show that it is a flat cofibration we will use
  Corollary~\ref{cor:incflatmor}.  In particular it suffices to show the
  following two statements:
  \begin{enumerate}
  \item If a simplex in $(L^\R_n K(\D_\dotp))$ is $j$-trivial (as a simplex in
    $K(\D_n)_m$) then it was already in $\Lambda_m (L^\R_n K(\D_\dotp))$---i.e. that
    it is $j$-trivial.  This is tautologically true.
  \item If a simplex in $K'(\C_n)_m$ is $j$-trivial when considered as a simplex
    in $K'(\D_n)_m$ then it is $j$-trivial as a simplex in $K'(\C_n)_m$.  This
    is also true, since $m$-triviality depends only on the underlying
    categorical structure of each simplex.
  \end{enumerate}
  Taken together, these two statements give exactly the condition in
  Corollary~\ref{cor:incflatmor}, so we see that $\eta_n$ is flat, as desired.
\end{proof}

\section{Increasing spectra vs flat spectra}\label{sec:inc->flat}
In this section we give the technical ingredients and the main steps of the
proof that increasing spectra are flat.  Checking that the objects and maps are
well-defined is lengthy and annoying, and is thus postponed to
Section~\ref{sec:tech:inc->flat}. 
\subsection{An explicit model of the smash product} \label{app:smash}

Here we give an explicit construction of the smash product of two symmetric
spectra. 

 In \cite[Section I.5]{schwedebook}, Schwede gives the following
 construction of the $n$-th space of the smash product of $X$ and $Y$:
$(X\sma Y)_n$ is the coequalizer of the two maps
\[\alpha,\beta: \bigvee_{k+m=n-1} \Sigma_n^+ \sma_{\Sigma_k \times
    \Sigma_1 \times \Sigma_m} X_k \sma S^1 \sma Y_m \rto \bigvee_{k+m=n}
  \Sigma_n^+ \sma_{\Sigma_k\times \Sigma_m} X_k \sma Y_m,\]
where $\alpha_{k,m}: \Sigma_n^+ \sma_{\Sigma_k \times \Sigma_1\times \Sigma_m}
X_k \sma S^1 \sma Y_m \rto \Sigma_n^+ \sma_{\Sigma_{k+1}\times \Sigma_m}$
is induced by the structure map $X_k \sma S^1 \rto X_{k+1}$, and $\beta_{k,m}:
\Sigma_n^+ \sma_{\Sigma_k\times \Sigma_1 \times \Sigma_m} X_k \sma S^1 \sma Y^m
\rto \Sigma_n^+ \sma_{\Sigma_k \times \Sigma_{m+1}} X_k \sma Y_{1+m}$ is the map
induced by
\[S^1 \sma Y_m \rto^\cong Y_m \sma S^1 \rto^\sigma Y_{m+1} \rto^\cong Y_{1+m}.\]

We give a slightly different presentation of the same coequalizer.  Let
$\sT_n$ be the poset with
\begin{description}
\item[objects] pairs $(a,b)$ of nonnegative integers such that $a+b \leq n$,
  and
\item[morphisms] there is a unique morphism $(a,b) \rto (a',b')$ if $a \leq a'$
  and $b \leq b'$.
\end{description}
We define a functor $P_{n,X,Y}:\sT_n \rto s\Set$ by
\[(a,b) \rgoesto \Sigma_n^+ \sma_{\Sigma_a \times \Sigma_{n-a-b} \times
    \Sigma_b} X_a \sma S^{n-a-b} \sma Y_b.\]
A morphism $(a,b) \rto (a',b')$ goes to the morphism
\begin{align*}
  X_a \sma S^{n-a-b} \sma Y_b \rto^\cong X_a \sma S^{a'-a} \sma S^{n-a'-b'} \sma
  S^{b'-b} \sma Y_b &\rto^\cong (X_a \sma S^{a'-a}) \sma S^{n-a'-b'} \sma (Y_b \sma
  S^{b'-b})  \\ &\setlen{4em}{\rto^{\sigma \sma 1 \sma \sigma} X_{a'} \sma S^{n-a'-b'} \sma
  Y_{b+(b'-b)}} \\ &\setlen{6em}{\rto^{1\sma 1\sma \mathrm{tw}_{b,b'-b}} X_{a'} \sma S^{n-a'-b'}
  \sma Y_{b'}},
\end{align*}
where $\mathrm{tw}_{p,q}$ is the shuffle that swaps the first $p$ elements past the
last $q$ elements.

Let $\sT_n^{\leq 1}$ be the full subcategory of those pairs $(a,b)$ with $n-a-b
\leq 1$.

\begin{lemma}
  The inclusion $\sT_n^{\leq 1} \rto \sT_n$ is a final functor.
\end{lemma}

\begin{proof}
  It suffices to check that for all $(a,b)\in \sT_n$, the category
  $\C = (a,b)/\sT_n^{\leq 1}$ is connected.  This is the category of all pairs
  $(x,y)$ such that $n-1 \leq x+y \leq n$ and $x \geq a$, $y \geq b$.  This is
  nonempty because it contains the pair $(a,n-a)$.  Any $(x,y)$ with $x+y = n-1$
  has a morphism to $(x,n-x)$, so it suffices to check that all such pairs are
  connected.  However, the pair $(x,n-x)$ (with $x \leq n-b-1$) is connected to
  $(x+1,n-x-1)$ via the zigzag $(x,n-x) \lto (x,n-x-1) \rto (x+1,n-x-1)$.  Thus
  the category is connected, as desired.
\end{proof}

By definition,
\[\colim_{\sT^{\leq 1}_n} P_{n,X,Y} = \coeq(\alpha,\beta).\]
Thus Schwede's definition of $(X\sma Y)_n$ is exactly $\colim_{\sT^{\leq 1}_n}
P$.  In this paper we will instead use the definition
\begin{equation}\label{eq:explicit_smash}
(X\sma Y)_n = \colim_{\sT_n} P_{n,X,Y}.
\end{equation}
Since the inclusion of $\sT^{\leq 1}_n$ is final, these two are naturally
isomorphic, but our definition will be more useful to us in our analysis of
latching spaces.

We can therefore conclude that
  \begin{equation}\label{eq:lt_as_colim}
L^T_nX\cong\colim P_{n,X,\redS}.
\end{equation}

\begin{lemma}\label{lem:vn_via_colim}
The maps 
\begin{equation}\label{eq:first_map_in_vn}
\Sigma_n^+ \sma X_a \sma S^{n-a-b} \sma S^b\to X_n
\end{equation}
given by $(\gamma, x, \delta, \delta')\mapsto\gamma\cdot
\sigma(x,(\delta,\delta')) $ induce a map
\[\nu^T_n\colon L^T_nX\cong \colim_{\sT_n} P_{n,X,S}\to X_n\]
as in \eqref{eq:nu}.
\end{lemma}

\begin{proof}
The maps in \eqref{eq:first_map_in_vn} descend to 
\begin{equation}\label{eq:vn_via_colim}
	\Sigma_n^+ \sma_{\Sigma_a \times \Sigma_{n-a-b} \times
    \Sigma_b} X_a \sma S^{n-a-b} \sma S^b\to X_n
\end{equation}
since $\sigma$ is equivariant.

  Define
\begin{equation}\label{eq:maps_to_check}
  \begin{aligned}
    \alpha_k &= P_{n,X,\redS}\big((k,1,n-k-1) \rto (k+1,n-k-1)\big) \\
    \beta_k &= P_{n,X,\redS}\big((k,1,n-k-1) \rto (k,n-k)\big).\\ 
  \end{aligned}
\end{equation}
  It suffices to check that
  $\varphi_k\beta_k = \varphi_{k+1}\alpha_k$ for all $k$.  (It is not necessary
  to check the other relations from $\sT_n$ because the inclusion of
  $\sT^{\leq 1}_n$ into $\sT_n$ is final.)  

  The two composites are 
\[(\gamma, x, \delta, \delta')\mapsto (\gamma, \sigma(x,\delta), \delta') 
\mapsto \gamma\cdot \sigma(\sigma(x,\delta),\delta')) \]
and 
\[(\gamma, x, \delta, \delta')\mapsto (\gamma, x, \sigma(\delta,
  \delta'))\mapsto\gamma\cdot \sigma(x,\sigma(\delta,\delta')), \]
which are equal, as desired.
\end{proof}

\begin{remark}
  The image of $L_n^TX$ inside $X$ is exactly $\Lambda_nX$, by
  Lemma~\ref{lem:nuL=lambda}.  Thus in particular, $\nu_n^T$ is a cofibration if
  and only if it induces an isomorphism $L_n^TX \rto \Lambda_n X$.
\end{remark}

\subsection{Flat spectra are increasing}\label{sec:flat_to_increasing}

In this subsection we prove that all flat spectra are increasing.

\begin{proposition} \label{prop:flat->inc}
  All flat spectra are increasing.
\end{proposition}

\begin{proof} 
  Let $X$ be a flat spectrum.  We must prove, for all $n \geq 1$, that
  \begin{enumerate}
  \item 
  the structure map $\sigma: X_{n-1}\sma S^1 \rto X_n$ is monic, and
\item 
  for any $\rho\in \Sigma_n$ such that $\rho(n) \neq n$, and every
  $y,y'\in X_{n-1}$, if there exist $\delta,\delta'\in S^1$ such that 
  $\sigma_1(y,\delta) = \rho\cdot \sigma_1(y',\delta')$ then there exists a
  simplex $z\in X_{n-2}$ such that $y$ and $y'$ are latched to $z$.
  \end{enumerate}
  By \cite[Proposition III.2.9]{schwedebook}, for any flat spectrum $X$, the
  morphism $\lambda_X: S^1 \sma X \rto \operatorname{sh}X$ is a flat
  cofibration.  In particular, by \cite[Corollary III.3.12]{schwedebook} this
  implies that it is a level cofibration.  The components of this morphism are
  exactly the structure maps of $X$, so this implies that the structure maps of
  $X$ are monic, as desired.  This implies \eqref{it:structure->monic}.

  Now consider \eqref{it:latched}.   Using \eqref{eq:explicit_smash},
  $L^T_nX=\colim P_{n,X,\redS}$.  The fact that $X$ is flat is the statement
  that for all $n$, the map
  \begin{equation}\label{eq:nu_T_again}\nu_n^T:\colim
    P_{n,X,\redS} \rto X_n\end{equation}
  induced by the structure maps of $X$ is monic.

  Suppose that
  $x\in \sigma(X_{n-1}\sma S^1) \cap \rho\cdot \sigma(X_{n-1}\sma S^1)$.  This
  means that $x = \sigma(y,\delta) = \rho\cdot \sigma(y',\delta')$ for some
  $\rho$ such that $\rho(n) \neq n$.  Write $\rho = \gamma(\rho',1)$ where
  $\gamma$ is an $(n-1,1)$-shuffle and
  $(\rho',1)\in \Sigma_{n-1}\times \Sigma_1$.   (Since $\rho(n)\neq n$, $\gamma$ is nontrivial.) Then
  $\rho\cdot \sigma(y',\delta') = \gamma\cdot \sigma(\rho'\cdot y',\delta')$; in
  particular, by taking $y''=\rho'y'$, we have $x=\sigma(y,\delta)=\gamma \cdot \sigma (y'',\delta')$. 
Note that $(1,y,\delta)$ and
  $(\gamma,y'',\delta')$ represent distinct simplices in $P_{n,X,\redS}(n-1,1)$,
  and the component
  \[f:P_{n,X,\redS}(n-1,1) \rto X_n\]
  of $\nu_n^T$ is not monic.

  Since the only atomic morphism in
  $\sT_n$ 
  with codomain $(n-1,1)$ is the morphism $(n-2,1) \rto (n-1,1)$, and the
  morphism in \eqref{eq:nu_T_again}
  is monic, this implies that for all $m \geq 0$, the induced map
  \[P_{n,X,\redS}(n-1,1)_m\smallsetminus P_{n,X,\redS}(n-2,1,1)_m \rto (X_n)_m\]
  is monic.  In particular, we see that at least one of the simplices
  $(1,y,\delta)$ and $(\gamma,y'',\delta')$ must be in the image of
  $P_{n,X,\redS}(n-2,1)_m \rto P_{n,X,\redS}(n-1,1)_m$.  Since the image is
  $\Sigma_n$-equivariant, this implies that both  $(1,y,\delta)$ and $(\gamma,y'',\delta')$ must be in the image.

  In particular, there exist
  $(\tau,z,\delta_z,\xi_z),(\tau',z',\delta_z',\xi_z')\in P_{n,X,\redS}(n-2,1)$ such that,
  under the induced map $P_{n,X,\redS}(n-2,1) \rto P_{n,x,\redS}(n-1,1)$ we have
  $(1,y,\delta) = (\tau,\sigma(z,\delta_z),\xi_z)$ and
  $(\gamma,y'',\delta') = (\tau',\sigma(z',\delta_{z}'),\xi_z')$.  In particular, we see that
  $\tau = 1$, $\tau' = \gamma$, $\delta=\xi_z$ and $\delta'=\xi_z'$, and we have
  \[\sigma(\sigma(z,\delta_z),\delta) =
    \gamma\cdot\sigma(\sigma(z',\delta_z'),\delta').\]
  Write $\gamma = (\rho'\times 1)(n-1\ 1)$, where $\rho'$ is an $(n-2,1)$-shuffle.
  The above equations shows that
  \[\sigma_2(z,(\delta_z,\delta)) = \sigma_2(\rho' \cdot z',(\delta',\delta'_z)).\]
  Since $\sigma_2$ is monic, this implies that $z = \rho'\cdot z'$; in
  particular, we see that  $y$ and $y'$ are both latched to $z$.  This proves that
  $X$ satisfies \eqref{it:latched}.
\end{proof}

\subsection{Roots}

In an increasing symmetric spectrum, latched and unlatched simplices behave
similarly to degenerate and nondegenerate simplices in a simplicial set, in that
every latched simplex is latched to a \emph{unique} unlatched simplex, which we
call its \emph{root}.  The goal of this section is to prove this fact.

We begin by showing that latching is a transitive relation on simplices.

Latching is a transitive relation on simplices of $X$.

\begin{lemma}\label{lem:latch_transitive}
  Latching is transitive.  In other words, let $k < m < n$, and let $z\in X_k$,
  $y\in X_m$ and $x\in X_n$. Let $\delta\in S^{m-k}$ and $\epsilon\in S^{n-m}$.
  If $y$ is $(\delta,\tau)$-latched to $z$ and $x$ is $(\epsilon,\rho)$-latched
  to $y$ then $x$ is $(\delta\sma \epsilon, \rho(\tau\times 1))$-latched to $z$.
\end{lemma}

\begin{proof}
  We have
  \begin{align*}
    x &= \rho\cdot \sigma_{n-m}(y,\epsilon) = \rho\cdot \sigma_{n-m}(\tau\cdot
        \sigma{m-k}(z,\delta), \epsilon) \\
    &= \rho(\tau\times 1) \sigma_{n-m}(\sigma_{m-k}(z,\delta),\epsilon) =
      \rho(\tau\times 1) \sigma_{n-k}(z, \delta\sma\epsilon),
  \end{align*}
  as desired.
\end{proof}

For a simplex $x$ in $X_n$,   let $\mathcal{P}_x$ be the poset of simplices $y$ that $x$ is latched to, under
  the ordering $y < y'$ if $y'$ is latched to $y$.  This poset is graded with
  $|y| = k$ if $y\in X_k$, and if $y \leq y'$ then $|y| = |y'|$ if and only if
  $y = y'$.

\begin{lemma}\label{lem:pre:rootunique}
If simplices $y,y'$ have an upper bound
  $w$ in $\mathcal{P}_x$  then there exists a lower bound $z$ with $\max(|w|-|y|,|w|-|y'|) \geq
  \max(|y|-|z|,|y'|-|z|)$. 
\end{lemma}

\begin{proof}
 We prove this by induction on
  $\max(|w|-|y|,|w|-|y'|)$.

  If  $\max(|w|-|y|,|w|-|y'|)=0$, so that $w = y =y'$ then the claim trivially holds with $z =
  w$.

  Suppose  $\max(|w|-|y|,|w|-|y'|)=1$.  Then $z$ exists by
  Definition~\ref{def:increasing}\eqref{it:latched}.

  Now suppose that the claim holds for all pairs with this measure at most
  $M-1$, and consider a case where $|w|-|y| = M$ and $|w|-|y'| \leq M$.  Let
  $y = y_0 < y_1 <\cdots < y_{M} = w$ and
  $y' = y_0' < \cdots < y'_{|w|-|y'|} = w$, which exist by the definition of
  latching.  

First, suppose that $|w|-|y'| < M$.  Then the inductive hypothesis
  applies to $y_{M-1},y'$, so there exists $z' \leq y_{M-1},y'$ with
  $|y_{M-1}|-|z'|$ and $|y'|-|z'| < M$; in particular, $|y'|-|z'| = 1$ and
  $|y_{M-1}|-|z'| = |w|-|y'| < M$.  
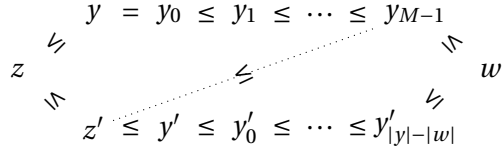
\[\begin{tikzpicture}
\node (w) at (.5,0) {$w$};
\node (y0) at (0,.375) [rotate =-45] {$\leq$};
\node (y1) at (-.5,.75) {$y_{M-1}$};
\node (y2) at (-1.25,.75) {$\leq$};
\node (y3) at (-1.75,.75) {$\ldots$};
\node (y4) at (-2.25,.75)  {$\leq$};
\node (y5) at (-2.75,.75) {$y_{1}$};
\node (y6) at (-3.25,.75)  {$\leq$};
\node (y7) at (-3.75,.75) {$y_{0}$};
\node (y8) at (-4.25,.75)  {$=$};
\node (y9) at (-4.75,.75) {$y$};
\node (y10) at (-5.25,.375) [rotate =45] {$\leq$};
\node (z0) at (-.25,-.375) [rotate =45] {$\leq$};
\node (z1) at (-.5,-.75) {$y'_{|y|-|w|}$};
\node (z2) at (-1.25,-.75) {$\leq$};
\node (z3) at (-1.75,-.75) {$\ldots$};
\node (z4) at (-2.25,-.75)  {$\leq$};
\node (z5) at (-2.75,-.75) {$y'_{0}$};
\node (z6) at (-3.25,-.75)  {$\leq$};
\node (z7) at (-3.75,-.75) {$y'$};
\node (z8) at (-4.25,-.75)  {$\leq$};
\node (z9) at (-4.75,-.75) {$z'$};

\node (z10) at (-5.25,-.375) [rotate =-45] {$\leq$};
\node (w1) at (-5.75,0) {$z$};
\draw[dotted](z9)--(y1) node[midway,rotate =25]{$\leq$};
\end{tikzpicture}
\]
But then
  $\max(|y_{M-1}|-|y|,|y_{M-1}|-|z'|) < M$, so the inductive hypothesis applies here;
  thus there exists a lower bound $z$ for $y$ and $z'$ such that
  $\max(|y|-|z|,|z'|-|z|) \leq \max(|y_1|-|y|,|y_1|-|z'|) < M$.  Thus $z$ is
  a lower bound for $y$ and $y'$ with $|y|-|z| < M$ and $|y'|-|z| \leq M$, as
  desired.

  Now suppose that $|w|-|y'| = M$.  By the base case applied to $y_{M-1}$ and
  $y'_{M-1}$ there exists $z'$ with $|z'| = |w|-2$ which is a lower bound for
  both.  Then the previously-handled case applies to the pair $y$ and $z'$
  (which have a common upper bound of $y_{M-1}$, giving a lower bound $z''$ with
  $|y|-|z''| = 1$ and $|z'|-|z''| \leq M-1$).  Then the previous case also applies
  to the pair $z''$ and $y'$ (which have a common upper bound of $y'_{M-1}$,
  giving a lower bound $z$ with $|z''|-|z| \leq M-1$ and $|y'|-|z| \leq M$).
  Since $|y|-|z''| = 1$, this implies that $z$ is a lower bound with $|y|-|z|
  \leq M$, as well, as desired.
\end{proof}

\begin{lemma} \label{lem:rootunique}
  If $X$ is increasing then for every simplex $x\in X_n$ there exists a unique
  maximal $k$, together with unique simplex $y\in X_{n-k}$ and $(\tau, \delta)$
  such that $x$ is $(\tau,\delta)$-latched to $y$.
\end{lemma}

\begin{proof}

  Since gradings must be nonnegative integers, Lemma~\ref{lem:pre:rootunique} implies the poset $\mathcal{P}_x$ must have a minimal
  element, and thus there is a unique simplex with maximal $k$.
\end{proof}

\begin{definition}\label{defn:root}
  For a simplex $x\in X_n$, its \emph{root} is the simplex $r_x\in X_{n-k}$,
  with maximal $k$, such that $x$ is latched to $r_x$.  

\end{definition}

Moreover, every simplex is \emph{uniquely} latched to its root.  In this way,
latched and unlatched simplices behave similarly to degenerate and nondegenerate
simplices in a simplicial set.  We show this by proving that the category below has an initial object.

  For a simplex $x$ in $X_n$,  let $\C_x$ be the category with
  \begin{description}
  \item[objects] simplices $y\in X_k$ for $k \leq n$ that $x$ is latched to, and
  \item[morphisms] for $y\in X_k$ and $y'\in X_{k'}$ for $k < k'$, a morphism $y
    \rto y'$ is a pair $(\tau,\delta)$ of a $(k,k'-k)$-shuffle $\tau$ and a
    $k'-k$-simplex $\delta$ such that $y'$ is $(\tau,\delta)$-latched to $y$,
    and
  \item[composition] given by composition of latching from Lemma~\ref{lem:latch_transitive}.
  \end{description}

\begin{lemma} \label{lem:pre:latchingunique}
  Given a solid-arrow diagram:
  \begin{diagram}
    { y &  z \\ & z'\\};
    \to{1-1}{1-2}^{(\tau,\delta)}
    \to{2-2}{1-2}_{(\epsilon,\delta_1)}
    \diagArrow{->,densely dotted}{1-1}{2-2}
  \end{diagram}
  in $\mathcal{C}_x$ where $|z'| = |z|-1$ there is a dotted arrow making the diagram
  commute. 
\end{lemma}

\begin{proof}
   We prove this by induction on $|z|-|y|$.
  If $|z|-|y|=1$ then both
  $\tau$ and $\epsilon$ are $(k,1)$-shuffles.  If $\tau(k+1) \neq \epsilon(k+1)$
  then $\tau^{-1}\epsilon(k+1)\neq k+1$, and thus by
  Definition~\ref{def:increasing}\eqref{it:latched} there exists $z''$ such that
  both $y$ and $z'$ are latched to $z''$.  But this contradicts the minimality
  of $y$.  Thus $\tau(k+1)=\epsilon(k+1)$, and $\tau = \epsilon$.  The diagram
  the states that
  $\epsilon\cdot \sigma(y,\delta) = \epsilon\cdot \sigma(z',\delta_1)$.  Since
  $\sigma$ is monic, this implies that $y = z'$ and $\delta = \delta_1$, so the
  dotted arrow trivially exists.

  Now suppose that the dotted arrow exists for $|z|-|y| < m$, and consider the
  case $|z|-|y|=m$.  Write $\tau = \epsilon' \circ \tau'$, where $\epsilon$ is
  an $(m+k-1,1)$-shuffle and $\tau'\in \Sigma_{m+k-1}\times \Sigma_1$.  Write
  $\delta = (\delta',\delta_1')$, where $\delta_1'\in S^1$.  Then the diagram
  above extends to a diagram
  \begin{diagram}[4em]
    { y & \tau'\cdot \sigma(y,\delta') & z \\ & & z'\\};
    \to{1-1}{1-2}^{(\tau',\delta')}
    \to{1-2}{1-3}^{(\epsilon',\delta_1')}
    \to{2-3}{1-3}_{(\epsilon,\delta_1)}
  \end{diagram}
  If $\epsilon = \epsilon'$ then, since $\sigma$ is monic,
  $\tau' \cdot \sigma(y,\delta') = z'$ and $\delta_1' = \delta_1$, so that the
  arrow $(\tau',\delta')$ gives the desired arrow $y \rto z'$.  It is unique
  because $\sigma$ is monic.  

If $\epsilon \neq \epsilon'$ then, by analogous
  reasoning to the first paragraph of this proof, the diagram extends to the following diagram:
  \begin{diagram}[4em]
    { y & \tau'\cdot \sigma(y,\delta') & z \\ & z'' & z'\\};
    \to{1-1}{1-2}^{(\tau',\delta')}
    \to{1-2}{1-3}^{(\epsilon',\delta_1')}
    \to{2-3}{1-3}_{(\epsilon,\delta_1)}
    \to{2-2}{2-3} \to{2-2}{1-2}
  \end{diagram}
  where $|z''| = |z'|-1$.  In particular, by the induction hypothesis there
  exists an arrow $y \rto z''$ making the diagram commute.  Since the right-hand
  square commutes, this implies the existence of an arrow $y \rto z'$ making the
  diagram commute.
\end{proof}

\begin{lemma} \label{lem:latchingunique}
  Let $X$ be an increasing spectrum.
  Let $x\in X_n$ and let $y\in X_k$ be the root of $x$.  If $(\tau,\delta)$ and
  $(\tau',\delta')$ are such that $x$ is $(\tau,\delta)$- and
  $(\tau',\delta')$-latched to $y$, then $(\tau,\delta) = (\tau',\delta')$.
\end{lemma}

\begin{proof}

  Our goal is to show that the root of $x$ is the initial object of $\C_x$.
  There is a functor from $\C_x$ to the poset $\mathcal{P}_x$ 
given by the forgetful functor.  Thus functor is
  full and injective on objects; in particular, it follows that $\C_x$ has an
  object $y$ which is ``minimal,'' in the sense that for all $z\in \C_x$, there
  exists a morphism $y \rto z$.  In order to show that $y$ is initial, it
  therefore remains to show that this morphism is \emph{unique}.

  We prove that for all $z\in \C_x$ $\#\Hom(y,z) = 1$ by induction on
  $|z|-|y|$.  Since $y$ is minimal, $|z| \geq |y|$ for all $z\in \C_x$.  When
  $|z| - |y| = 0$ we must have $z = y$ by Lemma~\ref{lem:rootunique}.  Since the only latching of $y$ to
  itself is the trivial latching, the claim holds in this case.  

Now suppose
  that $|z| - |y| = 1$.  Suppose there exist two different latchings
  $(\tau,\delta)$ and $(\tau',\delta')$ of $z$ to $y$.  If $\tau = \tau'$ then
  we have $\sigma_1(y,\delta) = \sigma_1(y,\delta')$, and since $\sigma$ is
  monic, $\delta = \delta'$.  If $\tau\neq \tau'$ then $\tau^{-1}\tau(n) \neq n$
  (since both $\tau$ and $\tau'$ are $(n-1,1)$-shuffles they are uniquely
  determined by their value on $n$), so we can write
  $\tau^{-1}\tau = \gamma\cdot (\rho_{n-1},\rho_1)$, so that
  $\sigma(y,\delta) = \gamma \cdot \sigma(\rho_{n-1}\cdot y, \delta')$.  But
  this implies, by Definition~\ref{def:increasing}\eqref{it:latched}, that $y$
  is latched to a lower-degree simplex, contradicting the assumption that $y$ is
  minimal.  Thus this case cannot happen and we have $\#\Hom(y,z) = 1$ in this
  case.

  Now consider the case when $|z|-|y| > 1$, and suppose that the claim holds up
  to $|z|-|y|-1$.  Fix any $w$ with a morphism $w \rto z$ and $|w| = |z|-1$.
  Then, by Lemma~\ref{lem:pre:latchingunique}, any morphism $y \rto z$ factors through $w \rto z$.
  Since, by the induction hypothesis, $\#\Hom(y,w) = 1$, it must therefore be
  the case that $\#\Hom(z,w) \leq 1$.  Since we know that $\#\Hom(y,w) \leq 1$,
  this implies the desired result.
\end{proof}

  We denote by
  $(\tau_x,\delta_x)$ the unique shuffle $\tau_x\in S_{n-k,k}$ and simplex
  $\delta_x\in S^k$ such that $x$ is $(\tau_x,\delta_x)$-latched to $r_x$.  If
  $x$ is unlatched we define $r_x = x$, $\tau_x = 1\in S_{n,0}$ and $\delta_x =
  *$.  

\subsection{The combinatorial latching space}

We now give an alternate combinatorial description of a latching space $\LP_nX$ using
roots.  Since roots are not well-defined for non-increasing spectra, this
definition \emph{only} holds when $X$ is increasing.  It takes significant work
to check that $\LP_n X$ is well-defined, so this will be postponed to
Section~\ref{sec:tech:inc->flat}. 

\begin{definition}
  Let $X$ be an increasing spectrum, and let $n \geq 0$.  Write $S_{(k,n-k)}
  \subseteq \Sigma_n$ 
  for the set of $(k,n-k)$-shuffles. The
  \emph{combinatorial latching space} $\LP_nX$ is the pointed
  $\Sigma_n$-simplicial set with $m$-simplices
  \[(\LP_nX)_m = \left\{ (\gamma, x, \delta) \in \bigvee_{k=0}^{n-1}
      S_{(k,n-k)}^+\sma (X_{k})_m \sma S^{n-k}_m \,\middle| x\, \hbox{
        unlatched}\right\}.\]
\end{definition}

The proof that this is well-defined is in
Proposition~\ref{prop:latch_well_defined}.

We can define a map 
\begin{equation}\label{eq:var_nu}
  \nu_n^C\colon \LP_nX\to X_n
\end{equation}
by $(\gamma, x, \delta)\mapsto \gamma\cdot \sigma(x,\delta)$.  There is also a
comparison map $\mu_n: L_n^TX \rto \LP_nX$ sending each simplex in the latching
space to its root and the latching data.  (The formal definition is more
complicated, and can be found on page~\pageref{def:mun}.)

\begin{proposition} \label{prop:latch_maps_commute}
  For an increasing spectrum $X$ the diagram
\begin{equation}\label{eq:latch=latch}\begin{tikzcd}
	{L_n^TX} && {\LP_nX } \\
	& {X_n}
	\arrow["{\mu_n}", from=1-1, to=1-3]
	\arrow["{\nu^T_n}"', from=1-1, to=2-2]
	\arrow["{\nu^C_n}", from=1-3, to=2-2]
      \end{tikzcd}\end{equation}
    commutes.
  \end{proposition}

  That $\mu_n$ and $\nu_n^C$ are well-defined is non-trivial, and is also
  postponed to Section~\ref{sec:tech:inc->flat}; see page~\pageref{proof:latch_maps_commute} for those proofs and the proof of this proposition.
  The key property of $\mu_n$ is that, for increasing spectra, it is an
  isomorphism:

\begin{proposition} \label{prop:latch=latch}
  When $X$ is an increasing spectrum, the map  $\mu_n:L_n^TX\rto \LP_nX $  is an isomorphism.
\end{proposition}

The proof of this is on page~\pageref{proof:latch=latch}.

\subsection{Proofs of the main results}

We are now ready to prove
Theorem~\ref{thm:characterize_flat_projective_spectra},
Corollary~\ref{cor:flatY->proj}, and
that increasing spectra are flat.

\begin{proof}[Proof of Theorem~\ref{thm:characterize_flat_projective_spectra}]\label{proof:characterize_flat_projective_spectra}  By Proposition~\ref{prop:flat->inc} $X$ is increasing.
Using Proposition~\ref{prop:latch=latch} we can consider the map $\nu^C$ rather than $\nu^T$.  Since $X$ is flat, $\nu^T$ is injective. Then the result follows from the characterization of projective cofibrations in Remark~\ref{rmk:characterize_cofib}. 
\end{proof}

\label{page_ref:flatY->proj}
\begin{proof}[Proof of Corollary~\ref{cor:flatY->proj}]
  If $Y$ is projective  for all $n,m\geq 0$, Theorem~\ref{thm:characterize_flat_projective_spectra} implies the action of
  $\Sigma_n$ on 
\[(Y_n)_m \smallsetminus (\Lambda_nY)_m\] 
is free.  
    Since $f$ is
  flat $\nu_n$ is a cofibration.  The action of $\Sigma_n$ on $(Y_n)_m \smallsetminus (\Lambda_nY)_m$ is free, so the action of $\Sigma_n$ on the subset 
\[(Y_n)_m \smallsetminus (\Lambda_nY \cup_{\Lambda_nX}
  f_n(X_n))_m\] is free.
\end{proof}

\begin{proposition} \label{prop:inc->flat}
  All increasing spectra are flat.
\end{proposition}

\begin{proof}
Using Proposition~\ref{prop:latch=latch} it suffices to show that $\nu_n^C$ is
injective.  We prove this by contradiction. 

Assume that there exist $m$-simplices $(\gamma,x,\delta)$ and
$(\gamma',x',\delta')$ with $x\in X_{m-k}$ and $x'\in X_{m-k'}$ such that
$\nu^C_n(\gamma,x,\delta) = \nu^C_n(\gamma',x',\delta')$.  Write $y =
\nu_n^C(\gamma,x,\delta)$.  Then by definition,
\[y = \gamma\cdot \sigma_k(x,\delta) = \gamma'\cdot \sigma_{k'}(x',\delta').\]
Thus $x$ is the root of $y$, with $y$ $(\gamma,\delta)$-latched to $x$.  But it
is also the case that $x'$ is the root of $y$, with $y$
$(\gamma',\delta')$-latched to $x'$.  By Lemma~\ref{lem:rootunique}, $x = x'$
and by Lemma~\ref{lem:latchingunique} $(\gamma,\delta) = (\gamma',\delta')$.
Thus $\nu_n^C$ is monic, as desired.
\end{proof}

\section{Technical details from
  Section~\ref{sec:inc->flat}} \label{sec:tech:inc->flat}

This section contains the proofs that were deferred from
Section~\ref{sec:inc->flat}: that $\LP_nX$, $\mu_n$ and $\nu^C$ are
well-defined, and have the relationships claimed in Propositions~\ref{prop:latch_maps_commute} and \ref{prop:latch=latch}.

\subsection{The 
combinatorial 
latching space}

The goal of this subsection is to prove that $\LP_n$ is well-defined.
\begin{proposition}\label{prop:latch_well_defined}
  For an increasing spectrum $X$, $\LP_nX$ is a pointed simplicial
  $\Sigma_n$-set with $m$-simplices $(\LP_nX)_m$.
\end{proposition}

The proof will require several helper lemmas.

\begin{lemma}\label{lem:latching_and_structure_maps}
if $y$ is $(\delta,\tau)$-latched to $x$ then $d_iy$ is
  $(d_i\delta,\tau)$-latched to $d_ix$ and $s_iy$ is $(s_i\delta,\tau)$-latched
  to $s_ix$ for all $i$.  If $x$ is unlatched then so is $s_ix$.
\end{lemma}

\begin{proof}
The first sentence follows by direct computation.  For the second
 suppose that $s_ix$ is
  $(\tau, \delta)$-latched to $y\in X_k$, so that $s_ix = \tau \cdot
  \sigma_{n-k}(y,\delta)$.  Applying $d_i$ to both sides, we see that
  \[x = d_is_i x = \tau \cdot d_i\sigma_{n-k}(y,\delta) = \tau \cdot
    \sigma_{n-k}(d_iy, d_i\delta).\]
  In particular, we see that $x$ is $(\tau, d_i\delta)$-latched to $d_iy$,
  contradicting the assumption that $x$ is unlatched.  
\end{proof}

\begin{lemma}\label{lem:root_of_d_s} For $i,j,i',j'$ such that $s_id_j =
  d_{j'}s_{i'}$ in the simplicial identities,
$s_i(r_{d_jx} )$ is the root of $d_{j'}s_{i'}x$ and  $d_{j'}s_{i'} x$ is $(\tau_{d_jx} , s_i(\delta_{d_jx} ))$-latched to $s_i(r_{d_jx} )$.
\end{lemma}

\begin{proof}
First note that 
  \[d_{j'}s_{i'} x = s_id_j x = s_i(\tau_{d_jx} \cdot \sigma_{n-k}(r_{d_jx} ,\delta_{d_jx} ) )= \tau_{d_jx} 
    \cdot \sigma_{n-k}(s_i(r_{d_jx} ),s_i(\delta_{d_jx} )),\] so $d_{j'}s_{i'} x$ is $(\tau_{d_jx} , s_i(\delta_{d_jx} ))$-latched to $s_i(r_{d_jx} )$.  
    If $s_i(r_{d_jx} )$ were
  latched to a simplex in a lower degree then  Lemmas~\ref{lem:latch_transitive} and \ref{lem:latching_and_structure_maps} imply $d_is_i(r_{d_jx} )=r_{d_jx} $ would be, as
  well. 
\end{proof}

We are now ready to prove Proposition~\ref{prop:latch_well_defined}.

\begin{proof}[Proof of Proposition~\ref{prop:latch_well_defined}]
By the second sentence in Lemma~\ref{lem:latching_and_structure_maps}, 
  the map  $s_i\colon (\LP_nX)_m \to (\LP_nX)_{m-1} $ given by
  \begin{equation}\label{lem:s_latching}
  s_i(\gamma,x,\delta) = (\gamma, s_i x, s_i\delta)\end{equation} is well defined.  

For $(\gamma, x, \delta)\in \LP_nX$, let $\gamma(\tau_{d_ix},1)=\gamma'(1,\rho)$ as in Lemma~\ref{lem:d_latching}(2).
Then we define 
\begin{equation}\label{eq:d_latching}
d_i(\gamma,x,\delta) = (\gamma', r_{d_ix}, \rho\cdot (\delta_{d_ix},d_i\delta))
\end{equation}
  If $d_ix = *$ then $d_i(\gamma,x,\delta) = *$.

  Let $i = j$ or $j+1$ and consider $d_is_j(\gamma,x,\delta)$.  Since $x$ is
  unlatched, by Lemma~\ref{lem:latching_and_structure_maps} so is $s_jx$.  Thus the root of $d_is_jx = x$ is $x$, and
  $\gamma' = \gamma$, $\tau = \id$, and $\delta$ is a $0$-simplex.  (See the second paragraph of Definition~\ref{defn:root}.)  In
  particular, the definition states that $d_is_j(\gamma,x,\delta) =
  (\gamma,x,\delta)$, as desired.  
  
    Consider $s_id_j(\gamma,x,\delta)$; let $i',j'$ be such that $s_id_j =
  d_{j'}s_{i'}$ in the simplicial identities.  Let 
  \[d_jx = \tau_{d_jx} \sigma_{n-k}(r_{d_jx} ,\delta_{d_jx} ),\] 
  and let $\gamma'$ and $\rho$ be
  such that 
  \begin{equation}\label{eq:defn_d_permutations}\gamma'(1,\rho) = \gamma(\tau_{d_jx} ,1)\end{equation}
   as in Lemma~\ref{lem:d_latching}(2).   We have
  \[s_id_j(\gamma,x,\delta) = s_i(\gamma', r_{d_jx} , \rho\cdot(\delta_{d_jx} , d_j\delta)) =
    (\gamma', s_i(r_{d_jx}) , \rho\cdot (s_i \delta_{d_jx} , s_id_j \delta)).\]
  On the other hand,
  \[d_{j'}s_{i'}(\gamma,x,\delta) = d_{j'}(\gamma, s_{i'}x,s_{i'}\delta).\] 
 Lemma~\ref{lem:root_of_d_s} and \eqref{eq:defn_d_permutations} give the middle equality in the following string
  \[d_{j'}s_{i'}(\gamma, x, \delta) = d_{j'}(\gamma, s_{i'}x, s_{i'}\delta) =
    (\gamma', s_i(r_{d_jx}), \rho\cdot (s_i\delta_{d_jx},d_{j'}s_{i'}\delta)) = (\gamma', s_i(r_{d_jx}),
    \rho\cdot(s_i\delta_{d_jx}, s_id_j\delta)),\]
  as desired.

  The relation about commuting degeneracies holds because they are applied
  coordinatewise.

  Lastly we consider the relation $d_id_j = d_{j-1}d_i$ for $i<j$.  The
  following equations are the definitions of all of the symbols on the right,
  where we omit the subscript on the $\sigma$'s so that we do not need to
  introduce more letters than necessary.  (So for example, the first equation
  states that $y$ is the root of $d_jx$, and it is $(\tau_j, \delta_j)$-latched
  to $d_jx$.)
  \[d_jx = \tau_j\cdot \sigma(y,\delta_j).\]
  \[d_iy = \tau_y \cdot \sigma(z, \delta_y).\]
  \[d_i x = \tau_i \cdot \sigma(y', \delta_i).\]
  \[d_{j-1}y' = \tau_{y'} \cdot \sigma(z', \delta_{y'}).\]

To fix notation, let $x\in X_\ell$, $z\in X_n$, $y'\in X_p$ and $z\in X_q$.  
  Using the notation above, 
  \begin{align*}
    d_id_j x &= d_i(\tau_j \cdot \sigma(y,\delta_j)) = \tau_j \cdot
    d_i(\sigma(y,\delta_j)) \\&= \tau_j \cdot \sigma(d_iy, d_i\delta_j) = \tau_j
    \cdot \sigma(\tau_y \cdot \sigma(z, \delta_y), d_i\delta_j) =
    \tau_j(\tau_y,1_{\ell-m}) \cdot \sigma(\sigma(z,\delta_y),d_i\delta_j) \\ &=
    \tau_j(\tau_y,1_{\ell-m}) \cdot \sigma(z, (\delta_y,d_i\delta_j)) = \tau_1 \cdot
                                                                       \sigma(z,
                                                                       \rho_1
                                                                       \cdot
                                                                       (\delta_y, d_i\delta_j))
  \end{align*}
  where  $\tau_1$ and $\rho_1$ satisfy 
\begin{equation}
\tau_j(\tau_y,1_{\ell-m})=\tau_1(1_{n},\rho_1).
\end{equation}
Similarly, 
  \begin{align*}
    d_{j-1}d_ix &= d_{j-1}(\tau_i\cdot \sigma(y', \delta_i)) = \tau_i\cdot
    \sigma(d_{j-1}y', d_{j-1}\delta_i) = \tau_i \cdot \sigma(\tau_{y'}\cdot
                  \sigma(z',\delta_{y'}), d_{j-1}\delta_i) \\
    &= \tau_i(\tau_{y'},1_{\ell-p}) \cdot \sigma(z', (\delta_{y'}, d_{j-1}\delta_i)) =
      \tau_2 \cdot \sigma(z', \rho_2 \cdot (\delta_{y'},d_{j-1}\delta_i))
  \end{align*}
  where  $\tau_2$ and $\rho_2$ satisfy 
\begin{equation}\label{eq:tauitauy}
\tau_i(\tau_{y'},1_{\ell-p})=\tau_2(1_{q},\rho_2).
\end{equation} 
Since roots are unique (Lemma~\ref{lem:rootunique})
  we see that 
\begin{align}
\tau_1 &= \tau_2, \label{eq:tau_s}
\\
z &= z',\label{eq:z_s}
\\ 
\rho_1\cdot (\delta_y, d_i\delta_j) &= \rho_2 \cdot (\delta_{y'}, d_{j-1}\delta_i)
\label{eq:rho_delta}
\\
n&=q
\end{align}

Take $\delta\in S^r$.
Then 
\begin{align*}
d_{j-1}d_i(\gamma,x,\delta) 
&= d_{j-1}(\gamma_3, y', \rho_3\cdot (\delta_{i},d_i\delta))
\\
&=(\gamma_4, z', \rho_4(\delta_{y'}, d_{j-1}(\rho_3\cdot (\delta_{i},d_i\delta))))
\\
&=(\gamma_4, z', \rho_4(1_{p-n}, \rho_3)(\delta_{y'}, d_{j-1}( \delta_{i},d_i\delta)))
\end{align*}
where 
 \begin{align}\label{eq:dj-1di_1}
\gamma(\tau_{i},1_{r})&=\gamma_3(1_p,\rho_3)
\\\label{eq:dj-1di_2}
\gamma_3(\tau_{y'},1_{(\ell+r)-p})&=\gamma_4(1_q,\rho_4).
\end{align} 
Then \eqref{eq:dj-1di_1} and \eqref{eq:dj-1di_2} imply
\[
\gamma(\tau_{i},1_{r})=\gamma_4(1_q,\rho_4)(\tau_{y'},1_{(\ell+r)-p})^{-1}(1_p,\rho_3)
\]
since $\tau_{y'}\in\Sigma_p$, 
\[
\gamma(\tau_{i},1_{r})(\tau_{y'},1_{(\ell+r)-p})=\gamma_4(1_q,\rho_4)(1_p,\rho_3).
\]
Then \eqref{eq:tauitauy} implies 
\[
\gamma(\tau_2,1_{r})(1_q,\rho_2,1_{r})=\gamma_4(1_q,\rho_4)(1_p,\rho_3)
\]
or 
\begin{equation}\label{eq:gammatau2}
\gamma(\tau_2,1_{r})=\gamma_4(1_q,\rho_4)(1_p,\rho_3)(1_q,\rho_2,1_{r})^{-1}.
\end{equation}
Lemma~\ref{lem:d_latching}(2) implies $\gamma_4$ and $\rho_4(1_{p-q},\rho_3)(\rho_2,1_{r})^{-1}$ are uniquely determined by $\gamma$ and $\tau_2$.

Similarly 
\begin{align*}
d_{i}d_j(\gamma,x,\delta) 
&= d_{i}(\gamma_5, y, \rho_5\cdot (\delta_{j},d_j\delta))
\\
&=(\gamma_6, z, \rho_6\cdot (\delta_y, d_i(\rho_5\cdot (\delta_{j},d_j\delta))))
\\
&=(\gamma_6, z, \rho_6(1_{m-q},\rho_5)\cdot (\delta_y, d_i (\delta_{j},d_j\delta)))
\end{align*}
where 
\begin{align*}
\gamma(\tau_j,1_r)&=\gamma_5(1_m, \rho_5)
\\
\gamma_5(\tau_y,1_{(\ell+r)-m})&=\gamma_6(1_n,\rho_6).
\end{align*}  
Parallel to the computation above
\begin{equation}\label{eq:gammatau1}
\gamma(\tau_1,1_{r})=\gamma_6(1_n,\rho_6)(1_m,\rho_5)(1_n,\rho_1,1_{r})^{-1}.
\end{equation}
Then Lemma~\ref{lem:d_latching}(2), \eqref{eq:tau_s}, \eqref{eq:gammatau2}, and \eqref{eq:gammatau1} imply 
\begin{align}
\gamma_4&=\gamma_6\label{eq:gamma_4_gamma_6}
\\
\label{eq:rho_s}
\rho_4(1_{p-q},\rho_3)(\rho_2,1_{r})^{-1}&=\rho_6(1_{m-n},\rho_5)(\rho_1,1_{r})^{-1}.
\end{align}

Postcomposing, 
\eqref{eq:rho_s}
gives 
\[
 \rho_4(1_{p-n}, \rho_3)(\delta_{y'}, d_{j-1}( \delta_{i},d_i\delta))
=\rho_6(1_{m-n},\rho_5)(\rho_1,1_{r})^{-1}(\rho_2,1_{r})(\delta_{y'}, d_{j-1}( \delta_{i},d_i\delta)).
\]
Applying $d_{j-1}$ to both terms on the far right gives 
\[\rho_6(1_{m-n},\rho_5)(\rho_1,1_{r})^{-1}(\rho_2,1_{r})(\delta_{y'}, (d_{j-1} \delta_{i},d_{j-1}d_i\delta)).\]
The equality in \eqref{eq:rho_delta} and $d_{j-1}d_i=d_id_j$ gives 
\[\rho_6(1_{m-n},\rho_5)(\rho_1,1_{r})^{-1}(\rho_1,1_{r})(\delta_{y}, (d_{i} \delta_{j},d_{i}d_j\delta)).\]
Then 
\begin{equation}\label{eq:third_term}
\rho_4(1_{p-n}, \rho_3)(\delta_{y'}, d_{j-1}( \delta_{i},d_i\delta))
=\rho_6(1_{m-n},\rho_5)(\delta_{y},d_i ( \delta_{j},d_j\delta))
\end{equation}
Finally \eqref{eq:gamma_4_gamma_6}, \eqref{eq:z_s}, and \eqref{eq:third_term} show
\[d_{j-1}d_i(\gamma,x,\delta) = d_{i}d_j(\gamma,x,\delta). \]
\end{proof}

\subsection{Comparisons of latching spaces}

This subsection proves Propositions~\ref{prop:latch_maps_commute} and \ref{prop:latch=latch}.

The map $\mu_n: L_n^T X \rto \LP_n X$ is a map out of a colimit, so in order to
construct it it suffices to construct a cocone under the colimit diagram.

\begin{lemma}\label{lem:define_phi} \label{def:mun}
  For all $0\leq k < n$, the  map
  \begin{equation}\label{eq:define_phi}\varphi_k: \Sigma^+_n \sma_{\Sigma_k \times \Sigma_{n-k}} X_k \sma S^{n-k} \rto \LP_nX
\end{equation}
  given by
  $\varphi_k (\gamma, x, \delta)= (\gamma', r_x, \rho\cdot
  (\delta_x, \delta))$ where 
 $\gamma\circ (\tau_x,1) =
  \gamma'\circ (1,\rho)$ as in Lemma~\ref{lem:d_latching}(2) 
    is well defined and the following diagram commutes.
\begin{equation}\label{eq:vnT_diagram}
\begin{tikzcd}
	{\Sigma^+_n \sma_{\Sigma_k \times \Sigma_{n-k}} X_k \sma S^{n-k} } && {\LP_nX} \\
	& {X_n}
	\arrow["{\varphi_k}", from=1-1, to=1-3]
	\arrow["{\eqref{eq:vn_via_colim}}"', from=1-1, to=2-2]
	\arrow["{\nu_n^C}"
, from=1-3, to=2-2]
\end{tikzcd}
\end{equation}
In particular, these maps induce the map $\mu_n: L_n^T X \rto \LP_nX$.
\end{lemma}

\begin{proof}
	By Lemma~\ref{lem:rootunique} $r_x$ and $\delta_x$ are unique.  By Lemma~\ref{lem:d_latching}(2) $\gamma'$ and $\rho$ are unique. So we have a map 
  $\Sigma^+_n \sma X_k \sma S^{n-k} \rto \LP_nX$.

  Let $(\upsilon, \upsilon')\in \Sigma_k \times \Sigma_{n-k}$.  
  Then 
  \begin{align*}
  	\gamma\circ (\upsilon, \upsilon')\circ (\tau_x,1_{n-k}) &=\gamma \circ (\upsilon\circ \tau_x, \upsilon')
	\\
	&=\gamma'\circ (1_\ell,\rho)\circ (\tau_x^{-1},1_{n-k})\circ  (\upsilon\circ \tau_x, \upsilon')
	\\
	&=\gamma'\circ (1_\ell,\rho)\circ   (\tau_x^{-1}\circ \upsilon\circ \tau_x, \upsilon')
	\\
	&=\gamma'\circ (1_\ell,\rho)\circ   (\tau_x^{-1}\circ \upsilon\circ \tau_x, 1_{n-k})\circ (1_k,\upsilon')
  \end{align*}
  and
  \begin{equation}\label{eq:phi_equiv_1}
\begin{aligned}\varphi_k (\gamma\circ (\upsilon, \upsilon'), x, \delta)
  &= 
  (\gamma'\circ (1_\ell,\rho)\circ   (\tau_x^{-1}\circ \upsilon\circ \tau_x, 1_{n-k}), r_x,  (1_{k-\ell},\upsilon')(\delta_x,\delta))
  \\
    &= 
  (\gamma'\circ (1_\ell,\rho)\circ   (\tau_x^{-1}\circ \upsilon\circ \tau_x, 1_{n-k}), r_x, (\delta_x,\upsilon'\delta))
\end{aligned}
  \end{equation}
  
  By definition  $\tau_x\sigma(r_x,\delta_x)=x$, 
  and so $\upsilon x=\upsilon \tau_x\sigma(r_x,\delta_x)$.  Then 
  \begin{align*}
  \gamma\circ (\upsilon\tau_x,1_{n-k})
  &= \gamma' \circ (1_\ell,\rho)\circ (\tau_x^{-1},1_{n-k})\circ (\upsilon\tau_x,1_{n-k})
  \\
  &=\gamma' \circ (1_\ell,\rho)\circ (\tau_x^{-1}\upsilon\tau_x,1_{n-k})
  \end{align*}
  and 
  \begin{equation}\label{eq:phi_equiv_2}
\varphi_k (\gamma, \upsilon x,  \upsilon' \delta)
  =(\gamma' \circ (1_\ell,\rho)\circ (\tau_x^{-1}\upsilon\tau_x,1_{n-k}), r_x, (\delta_x, \upsilon' \delta)).\end{equation}
Then \eqref{eq:phi_equiv_1} and \eqref{eq:phi_equiv_2} implies $\varphi_k$ is defined on the quotient.

Since 
\begin{align*}
(\gamma,x,\delta)\mapsto (\gamma',r_x,\rho(\delta_x,\delta))&\mapsto
\gamma'\cdot \sigma(r_x, \rho(\delta_x,\delta))
\\
&=\gamma'(1,\rho)\cdot \sigma(r_x, (\delta_x,\delta))
\\
&=\gamma(\tau_x,1)\cdot \sigma(r_x, (\delta_x,\delta))
\\
&=\gamma \sigma(\tau_x\cdot r_x, (\delta_x,\delta))
\\
&=\gamma \sigma(\sigma(\tau_x\cdot r_x, \delta_x),\delta)
\\
&=\gamma \sigma(x,\delta)
\end{align*}
the diagram in \eqref{eq:vnT_diagram} commutes.
\end{proof}

\begin{lemma}\label{lem:equiv_unlatched}
  Every
  simplex in $P_{n,X,\redS}(k,n-k)$ is equivalent to a simplex whose
  $X$-component is unlatched.
\end{lemma}

\begin{proof}
  Suppose that $(\gamma,x,\delta) \in P_{n,X,\redS}(k,n-k)$, and
   $r_x\in X_{k-m}$ is the root of $x$, with $x$ $(\tau_x,\delta_x)$-latched to $r_x$.
Then   \[(\gamma\circ (\tau_x,1),r_x,\delta_x,\delta)\in
  P_{n,X,\redS}(k-m,n-k);\]and its image in $P_{n,X,\redS}(k,n-k)$ is
  $(\gamma,x,\delta)$.

	Use Lemma~\ref{lem:d_latching}(2) to write 
	\[\gamma\circ (\tau_x,1)= \gamma'\circ (1,\rho)\]
where $\gamma$ is a $(k,n-k)$-shuffle, $\tau_x$ is a $(k-m,k)$-shuffle, $\gamma'$ is a $(k-m, n-k+m)$-shuffle, and $\rho\in \Sigma_{n-k+m}$.
Then the image of
  \[(\gamma\circ (\tau_x,1),r_x,\delta_x,\delta)=(\gamma'\circ (1,\rho),r_x,\delta_x,\delta)\] 
in  $P_{n,X,\redS}(k-m,n-(k-m))$ is
  $(\gamma',r_x,\rho\cdot(\delta_x,\delta))$.  
\end{proof}

We are now ready to prove Proposition~\ref{prop:latch_maps_commute}.

\begin{proof}[Proof of Proposition~\ref{prop:latch_maps_commute}] \label{proof:latch_maps_commute}
To show the maps in \eqref{eq:define_phi} induce a map
\begin{equation}\label{eq:define_nu}
\mu_n:\colim P_{n,X,\redS} \rto \LP_nX
\end{equation}
it is enough to 
  check that
  $\varphi_k\beta_k = \varphi_{k+1}\alpha_k$ for all $k$ and $\alpha_k$ and $\beta_k$ as in \eqref{eq:maps_to_check}. 
  
  Fix a simplex
  $(\gamma, x, \delta, \delta')$ in
  $\Sigma^+_n \sma_{\Sigma_k \times \Sigma_1 \times \Sigma_{n-k-1}} X_k \sma S^1
  \sma S^{n-k-1}$ 
  then 
  \begin{align*}
  \varphi_{k+1}\alpha_k(\gamma, x, \delta, \delta')&=\varphi_{k+1}(\gamma, \sigma(\tau_x\cdot \sigma(r_x,\delta_x),\delta), \delta') 
  \\
  &= \varphi_{k+1}(\gamma,
    (\tau_x,1)\cdot \sigma(r_x, (\delta_x,\delta)), \delta')
   \\
   &= (\gamma', r_x, \rho\cdot (\delta_x, \delta,\delta'))
  \end{align*}
  and 
  \begin{align*}
  \varphi_k\beta_k(\gamma, x, \delta, \delta')&=\varphi_k(\gamma, x, (\delta,\delta'))
  \\
  &=(\gamma', r_x, \rho\cdot (\delta_x, \delta, \delta'))
  \end{align*}
  Thus, by the universal property of the colimit, the morphism $\mu_n$ is
  well-defined. 

Since the diagram in \eqref{eq:vnT_diagram} commutes, the diagram in \eqref{eq:latch=latch} 
commutes.
\end{proof}

\begin{lemma}\label{lem:injective_unlatched} If $X$ is increasing, 
 $\mu_n$ maps simplices whose $X$-component is unlatched to
  distinct simplices.
\end{lemma}

\begin{proof} 
  Fix $k$, and suppose that
  $\varphi_k(\gamma_1,x_1,\delta_1) = \varphi_k(\gamma_2,x_2,\delta_2)$, with
  $x_1,x_2$ unlatched.  Then
  \[\sigma_k(x_1,\delta_1) = \gamma_1^{-1}\gamma_2 \sigma_k(x_2,\delta_2).\]
    We
  begin by proving by contradiction that $\gamma_1 = \gamma_2$.  Suppose
  otherwise, and let $m$ be maximal such that $\gamma_1(m) \neq \gamma_2(m)$, so
  that $\gamma_1^{-1}\gamma_2(\ell) = \ell$ for $\ell > m$; we can then write
  \[\sigma_{n-m}\big(\sigma_{k-(n-m)}(x_1,\delta_1'),\delta_1''\big) =
    \sigma_{n-m}\big(\gamma_1^{-1}\gamma_2
    \sigma_{k-(n-m)}(x_2,\delta_2'),\delta_2''\big).\]
  (Here, $\delta_i'$ is the first $k-(n-m)$ coordinates of $\delta_i$, and
  $\delta_i''$ are the remaining coordinates.  By abuse of notation, we write
  $\gamma_1^{-1}\gamma_2$ for the restriction to the first $m$ coordinates.)
  Since $X$ is increasing, $\sigma_{n-m}$ is injective, and we conclude that
  $\delta_1'' = \delta_2''$ and that
  \[\sigma(x_1,\delta_1') = \gamma_1^{-1}\gamma_2 \sigma(x_2,\delta_2').\]
  Since $\gamma_1^{-1}\gamma_2(m) \neq m$, by
  \ref{def:increasing}\eqref{it:latched} there exists a simplex which both $x_1$
  and $x_2$ are latched to.  This is a contradiction, since we assumed that both
  $x_1$ and $x_2$ are unlatched.  Thus we must have $\gamma_1 = \gamma_2$, so
  that we can conclude that $\sigma_k(x_1,\delta_1) = \sigma_k(x_2,\delta_2)$.
  But since $X$ is increasing $\sigma_k$ is injective, and we see that
  $x_1 = x_2$ and $\delta_1=\delta_2$, as desired.

  Now suppose that we have $k_1 < k_2$, and $(\gamma_i,x_i,\delta_i)$ with $x_i$
  unlatched, such that
  \[\varphi_{k_1}(\gamma_1,x_1,\delta_1) =
    \varphi_{k_2}(\gamma_2,x_2,\delta_2).\] By definition, this states that
  $(\gamma_1,x_1,\delta_1) = (\gamma_2,x_2,\delta_2)$, since $x_i$ is its own
  root.  In particular, we see that the images of $\varphi_{k_1}$ and
  $\varphi_{k_2}$ are disjoint when restricted to roots.
\end{proof}

We are now ready to prove Proposition~\ref{prop:latch=latch}.

\begin{proof}[Proof of Proposition~\ref{prop:latch=latch}] \label{proof:latch=latch}
Notice that if we restrict
  $\mu$ to the $k=n-1$-component, 

  Now suppose that we have $k_1 < k_2$, and $(\gamma_i,x_i,\delta_i)$ with $x_i$
  unlatched, such that
  \[\varphi_{k_1}(\gamma_1,x_1,\delta_1) =
    \varphi_{k_2}(\gamma_2,x_2,\delta_2).\] By definition, this states that
  $(\gamma_1,x_1,\delta_1) = (\gamma_2,x_2,\delta_2)$, since $x_i$ is its own
  root.  In particular, we see that the images of $\varphi_{k_1}$ and
  $\varphi_{k_2}$ are disjoint when restricted to roots.
\end{proof}

\bibliographystyle{alpha}
\bibliography{CPZ}

\end{document}